\documentclass[12pt, titlepage]{article}
\usepackage{eurosym}
\usepackage[dvipsnames]{xcolor}
\usepackage{amsmath,amssymb,amsthm,amsfonts}
\usepackage[mathscr]{eucal}
\usepackage{fullpage}
\usepackage{graphicx}
\usepackage{subfig}
\usepackage{epstopdf}
\usepackage{lscape}
\usepackage{setspace}
\usepackage{epsfig}
\usepackage[round]{natbib}
\usepackage{multirow}
\usepackage{color}
\usepackage[latin1]{inputenc}
\usepackage{bbm}
\usepackage{mathtools}
\usepackage{appendix}
\usepackage{rotating}
\usepackage{cancel}
\usepackage[normalem]{ulem}
\usepackage{url}

\numberwithin{equation}{section} 

\newtheorem{theorem}{Theorem}[section]
\newtheorem{corollary}{Corollary}[section]

\newtheorem{proposition}{Proposition}[section]
\newtheorem{definition}{Definition}[section]
\newtheorem{remark}{Remark}
\newtheorem{example}{Example}

\def\XX{\boldsymbol{X}}
\def\xx{\boldsymbol{x}}

\def\II{\mathcal{I_P}}
\def\cc{\boldsymbol{c}}

\def\rr{\boldsymbol{r}}
\def\ee{\boldsymbol{e}}
\def\FF{\boldsymbol{F}}

\def\RRR{\boldsymbol{R}}

\def\MC{\mathcal{C}_{\mathcal{H}}}

\def\FFF{\mathcal{F}}

\def\EEE{\mathcal{E}_d}

\def\mm{\boldsymbol{m}}
\def\pp{\boldsymbol{p}}

\def\NN{\mathbb{N}}

\def\RR{\mathbb{R}}
\def\aa{\boldsymbol{a}}

\def\zz{\boldsymbol{z}}
\def\QQ{\mathbb{Q}}
\def\ff{\boldsymbol{f}}

\def\design{\mathcal{X}}
\def\DDD{\mathcal{D}_d}

\DeclareMathOperator\expval{E}

\DeclareMathOperator{\rank}{rank}

\begin{document}

\author{ R.FONTANA\footnote{%
Corresponding author: Roberto Fontana
Department of Mathematical Sciences G. Lagrange, Politecnico di
Torino. Email: roberto.fontana@polito.it}   \\ \textit{\ 
Department of Mathematical Sciences G. Lagrange,} \\ {Politecnico di
Torino.}
\\ P. SEMERARO\\ \textit{\ 
Department of Mathematical Sciences G. Lagrange,} \\ {Politecnico di
Torino.}}
\title{High dimensional Bernoulli distributions:\\
 algebraic representation and applications}
\maketitle

\begin{abstract}
The main contribution of this paper is to find a representation of  the class   $\FFF_d(p)$  of multivariate Bernoulli distributions with the same mean $p$   that allows us to  find its generators analytically in any  dimension.
We map  $\FFF_d(p)$ to an ideal of points and we prove that the class $\FFF_d(p)$ can be generated  from a finite set of simple polynomials. We present two applications. Firstly, we show that polynomial generators help to find extremal points of the convex polytope $\FFF_d(p)$ in high dimensions. Secondly,  we  solve the problem of determining the lower bounds in the convex order for sums of multivariate Bernoulli distributions with given margins, but with an unspecified dependence structure.

\end{abstract}


\section{Introduction}
%
%
%

This paper proposes an algebraic representation of the
 class $\FFF_d(p)$  of $d$-dimensional Bernoulli distributions with Bernoulli univariate marginal distributions with parameter $p$, - the class of joint distributions of $(X_1,\ldots, X_d)$, where $X_i\sim Bernoulli(p), \,\,\, i=1,\ldots, d$. This representation turns out to be fundamental for solving some open issues regarding this class.
 Our main theoretical contribution is  to find a set of generators of the class $\FFF_d(p)$ analytically  in any  dimension. We build on the geometrical structure of $\FFF_d(p)$.  The class $\FFF_d(p)$ is a convex polytope \cite{fontana2018representation} that   we map  into an ideal of points. Using the Gr\"oebner basis of the ideal  we find an analytical  set of polynomials that  generates the class $\FFF_d(p)$. These polynomial generators are also very simple  in high dimensions. Through two applications we show that this novel representation allows us to address and solve some open problems in applied probability and statistics.

One  open problem in applied probability is to find the  lower bounds in the convex order for sums $S=X_1+\ldots+X_d$ of the components of random vectors in a given Fr\'echet class, i.e.  with given margins and an unspecified  dependence structure (see, e.g. \cite{kaas2000upper}). The upper bound is known to be the upper Fr\'echet bound, while the problem to find the lower bound has been solved only  under specific and restrictive conditions. Sums of Bernoulli random variables have received much attention among researchers in
 computer science, optimization, engineering and finance because of its wide applicability, see e.g. \cite{padmanabhan2021tree} and references therein. The minimum convex order corresponds  to the minimum risk distribution and it is relevant for example in credit risk, where univariate marginal distributions  represent the probability of default of obligors in a credit portfolio (a classical reference for credit risk modelling  is \cite{mcneil2005quantitative}). If obligors belong to the same class of rating, they have the same marginal default probability. Therefore it is  important to assess the risk associated to possible dependence among obligors, especially nowadays since the global economy and risks are  highly interconnected.
 For sums of Bernoulli variables this problem  is solved only if $pd<1$, which  is a restrictive hypothesis for credit portfolios because banks usually handle hundreds of obligors.

The main application of  our theoretical results is to solve the problem of finding lower bounds for sums of the components of  random vectors with pmf in $\FFF_d(p)$ for any $p$ and $d$.
We proceed in two steps. First we find the minimal distribution of sums in the convex order by building on the geometrical structure of discrete distributions. In this first step  we immediately obtain one corresponding pmf in $\FFF_d(p)$, the pmf of exchangeable Bernoulli distributions (see \cite{fontana2020model}). To have a minimal pmf that is not-exchangeable is not trivial, and this is our second step.  We find one pmf in  $\FFF_d(p)$ that is minimal and not-exchangeable by combining what we will define as fundamental polynomials. The proof of this result shows how we can work with polynomials to obtain pmfs that satisfy some conditions. We are confident that this approach can be used to address many other issues.
We also prove that the minimum convex order corresponds to the minimal mean correlation. It is well known that the minimal convex order is associated to negative dependence. In fact for $pd<1$ the minimum convex order of sums is the lower Fr\'echet bound of the class, that is the distribution of the mutually exclusive random variables. Mutual exclusivity  was first studied by \cite{dhaene1999safest}  in the context of insurance and finance. It was  called the safest dependence structure and  was  characterized as the strongest negative dependence structure  by \cite{cheung2014characterizing}. We also propose  a generalization of mutual exclusivity for $pd>1$ and we  study  its  relationship with the minimum convex order, to attempt to find a general notion of the safest dependence structure.

Statistical research extensively  investigates classes of multivariate Bernoulli distributions and its properties (see e.g. \cite{dai2013multivariate}, \cite{marchetti2016palindromic} and \cite{chaganty2006range}) because of the importance of binary data in applications.
 Issues such as high dimensional simulation - see e.g.\cite{haynes2016simulating},\cite{kang2001generating} and \cite{shults2016simulating} - and estimate -\cite{kvam1996maximum}- are very important for any  application and our novel representation is particularly convenient to work in high dimensions.  High dimensional simulation and testing is possible for some classes and under some conditions, for example see \cite{qaqish2003family}, \cite{kang2001generating}, \cite{shults2016simulating}, \cite{emrich1991method}. High dimensional simulation for exchangeable Bernoulli pmfs is addressed in \cite{fontana2021exchangeable}. Exchangeable Bernoulli pmfs are points in a convex polytope whose extremal points are analytically known (\cite{fontana2020model}) and high dimensional simulations is possible because we know how to sample  from a polytope \cite{fontana2021exchangeable}. To use this approach extremal points are necessary.
 In \cite{fontana2018representation} the authors provide a method to explicitly find the extremal points of  $\FFF_d(p)$, but computational complexity increases very quickly and they can not be found in high dimensions.
 We prove that fundamental polynomials are associated to extremal points in the polytope, thus this representation also helps in finding analytically a huge number of extremal points in any dimension. We provide an algorithm. In this application it is evident that working with polynomials is simpler than with pmfs.

The paper is organized as follows. Section \ref{PrDef} introduces the setting and formalizes the problem. The main theoretical results are given in Section \ref{MainR}, where we provide an algebraic representation of $\FFF_d(p)$. Section \ref{algo} uses the algebraic representation to address the issue to find the extremal points and provides an algorithm to find  them in any dimension. Section \ref{lcx} applies the theoretical results  to solve the problem of finding bounds in the convex order for sums of  variables with pmf in $\FFF_d(p)$ if $pd>1$. We conclude this section proving that for $pd>1$ the minimal convex order corresponds to the minimal mean correlation. We also extend the notion of mutually exclusive random vectors to the case $pd>1$  and we study its connection with minimality in the convex order.
Section \ref{Conc} concludes.

\section{Preliminaries and motivation}\label{PrDef}

We build on the geometrical representation of the class $\FFF_d(p)$ as a convex polytope to map it into  a ring of polynomials. This approach requires the introduction of some notation from computational geometry, algebra and algebraic geometry, a standard reference for these topics is \cite{cox2013ideals}.

\subsection{\protect\bigskip Notation \label{not}}

Let $\FFF_d$ be the set of $d$-dimensional probability mass functions (pmfs) which have Bernoulli univariate marginal distributions. Let us consider the Fr\'echet class  $\FFF_d(p)\subseteq\FFF_d$ of pmfs  in $\FFF_d$ which have the same Bernoulli marginal distributions of parameter p, $B(p)$.
We assume throughout the paper that $p$ is rational, i.e.  $p\in \QQ$. Since $\mathbb{Q}$ is dense in $\RR$, this is not a limitation in applications.

If  $\XX=(X_1, \dots, X_d)$ is  a random vector with pmf in $\mathcal{F}_d$, we denote
\begin{itemize}
	\item its cumulative distribution function by $F$ and its mass function by $f$;
	\item the column vector which contains the values of $F$ and $f$ over $\design=\{0, 1\}^d$,  by $\FF= (F_1, \ldots, F_D)=(F_{\xx}:\xx\in\design):=(F(\xx):\xx\in\design)$, and $\ff = (f_1,\ldots, f_{D})=(f_{\xx}:\xx\in\design):=(f(\xx):\xx\in\design)$, $D=2^d$, respectively; we make the non-restrictive hypothesis that the set $\design$ of $2^d$ binary vectors is ordered according to the reverse-lexicographical criterion.
For example $\mathcal{S}_3=\{000, 100, 010, 110, 001, 101, 011, 111\}$;
\end{itemize}
The notations $\ff\in \FFF_d(p)$,  $f\in\FFF_d(p)$,   $\XX\in \FFF_d(p)$ and $\FF\in \FFF_d(p)$   indicate that $f\in\FFF_d(p)$, $\XX\in \FFF_d(p)$, $\XX$  has pmf $f\in \FFF_d(p)$ and  the cdf $F$ has pmf in $\FFF_d(p)$, respectively.

Given a vector $\xx$ we denote by $|\xx|$ the number of elements of $\xx$ which are different from $0$. If $\xx$ is binary, $|\xx|=\sum_{i=1}^dx_i$.
We assume that vectors are column vectors and we denote  $A^{\top}$ the transpose of a matrix $A$.
Given two matrices  $A\in \mathcal{M}(n\times m)$ and $ B\in \mathcal{M}(d\times l)$: \begin{itemize}
\item  the matrix $A\otimes B=((a_{ij}B)_{1\le i\le n, 1\le j\le m})\in \mathcal{M}(nd\times ml)$ indicates their Kronecker product and $A^{\otimes n}$ is $\underbrace{A\otimes \cdots \otimes A}_{n \text{ times}}$;
\item if $n=d$, $A||B$ denotes the row concatenation of $A$ and $B$;
\item  if $m=l$, $A//B$ denotes the column concatenation of $A$ and $B$;
\end{itemize}
%
Finally, we denote by $P(\xx)=\sum_{\alpha\in\mathcal{A}}a_{\alpha}\xx^{\alpha}$ a polynomial in $\mathbb{Q}[x_1,\ldots,x_d]$,  where $\xx^{\alpha}=x_1^{\alpha_1}\cdots x_d^{\alpha_d}$, $\alpha=(\alpha_1,\ldots,\alpha_d)\in\mathcal{A}$ and $\mathcal{A}$ is a proper set of multiindexes. We will often use $\mathcal{A}=\{0,1\}^{d}$.  In some cases we write $a_{\sup(\alpha)}$, where $\sup(\alpha)=\{i\in \{1,\ldots,d\}: \alpha_i\neq0\}$ to simplify the notation, i.e $a_{\{1,4\}}\equiv a_{1,4}$ instead of $a_{(1,0,0,1)}$ and $a_{\emptyset}$ instead of $a_{(0,\ldots, 0)}$. Finally, if no confusion arises $a_j, j=1,\ldots D$ stands for $a_{\alpha}$ where $\alpha\in\{0,1\}^{d}$ is the $j$-th term in reverse lexicographic order, i.e. $a_{10}$ instead of $a_{1,4}=a_{(1,0,0,1)}$.

\subsection{The convex polytope $\FFF_d(p)$}
Let $\XX$ be a multivariate Bernoulli random variable whose distribution belongs to the Fr\'echet class $\FFF_d(p)\subseteq\FFF_d$,  $\XX\in\FFF_d(p)$. We can write
 $$\expval(X_i) =\xx_i^{\top} \ff,$$
where $\xx_i$ is the vector which contains only the $i$-th element of $\xx \in \design$, $i\in\{1,\ldots,d\}$, e.g for the bivariate case $\xx_1^T=(0, 1,0,1)$ and $\xx_2^T=(0, 0,1,1)$.
For $\ff\in \FFF_d(p)$ we have
\begin{eqnarray*}
\begin{cases}
\xx_i^{\top} \ff=p \\
(\boldsymbol{1}-\xx_i)^{\top} \ff =q,
\end{cases}
\end{eqnarray*}
where $\boldsymbol{1}$ is the vector with all the elements equal to $1$ and $q=1-p$. If we consider the odds of the event $X_i=0$,  $c=\frac{1}{\gamma}=q/p$, we have
\[
((\boldsymbol{1}-\xx_i)^{\top} -c \xx_i^{\top}) \ff= 0.
\]
Let $H$ be  the $d \times D$ matrix whose rows, up to a non-influential multiplicative constant, are $((\boldsymbol{1}-\xx_i)^{\top} - c\xx_i^{\top})$, $i\in\{1,\ldots,d\}$.
The probability mass functions $\ff$ in $\FFF_d(p)$ are the positive normalized, i.e. $f_i\geq0$,  $\sum_{i=1}^{D}f_i=1$, solutions  of the linear system
\begin{equation}\label{LS}
H\ff=\boldsymbol{0}.
\end{equation}

From the standard theory of linear equations we know that all the positive, normalized
solutions of \eqref{LS} are the elements of  a  convex polytope, thus
\begin{equation}  \label{cone}
\FFF_d(p)=\{\ff\in \RR^{2^d}: H\ff=0, f_i\geq 0, \sum_{i=1}^Df_i=1\}.
\end{equation}
The solutions of the  linear system in Equation \ref{LS} can be generated as convex combinations of a set of
generators which are referred to as extremal points of the convex polytope. Formally, for any $\ff\in \FFF_d(p)$
there exist $\lambda_1\geq0,\ldots, \lambda_n\geq0$ summing up to one and extremal points $\rr\in \FFF_d(p)$ such that
\begin{equation*}
\ff=\sum_{i=1}^{n}\lambda_i\rr_i.
\end{equation*}

We call  $\rr_i$ extremal points or  extremal pmfs. We denote with $\RRR_i$ a random vector with distribution $\rr_i$.

The
proof of the following proposition follows from Lemma 2.3 in \cite{terzer2009large}.

\begin{proposition}
\label{multinulli} Let us consider the linear system
\begin{equation*}
A\boldsymbol{z}=0,\boldsymbol{z}\in \mathbb{R}^{m}  \label{system}
\end{equation*}%
where $A$ is a $n\times m$ matrix, $n\leq m$ and $\rank(A)=n$. The
extremal points of the polytope \eqref{cone} have at most $n+1$ non-zero
components.
\end{proposition}
\begin{example}\label{example}
As an illustrative example we consider the class $\FFF_3(2/5)$, i.e. $d=3$ and $p=\frac{2}{5}$. We have $c=\frac{3}{2}$ and
\begin{equation}
H=\left(
\begin{array}{cccccccccccccccc}
1&-\frac{3}{2}  &1&-\frac{3}{2}   & 1&-\frac{3}{2}  &1&-\frac{3}{2} \\
1 & 1&-\frac{3}{2} & -\frac{3}{2}& 1 & 1&-\frac{3}{2} & -\frac{3}{2}\\
 1 & 1& 1 & 1 & -\frac{3}{2} & -\frac{3}{2}& -\frac{3}{2}& -\frac{3}{2}
\end{array}%
\right).\label{H3}
\end{equation}%
 In this case the  extremal points can be found using 4ti2 (see \cite{fontana2018representation}).
Giving $H$ as input, we generate  the  matrix $R$ whose columns are  the extremal pmfs  $\rr_{i}$, $i=1,\ldots,9$, which are reported in Table \ref{tab:d}.

\begin{table}[h]
	\centering
		\begin{tabular}{ccc|rrrrrrrrr}
$\xx_1$ &	$\xx_2$ &	$\xx_3$ &	$\rr_1$ & 		$\rr_2$ & 		$\rr_3$ & 		$\rr_4$ & 		$\rr_5$ & 		$\rr_6$ &$\rr_7$&$\rr_8$& $\rr_9$\\
		\hline
0 &	0 &	0 &	$\frac{1}{5}$ &	$\frac{1}{5}$ &	$\frac{1}{5}$&	$\frac{2}{5}$ &	$\frac{3}{5}$&	0 &0&0&0\\
1 &	0 &	0 &	0 &	0 &	$\frac{2}{5}$ &	0 &	0 &	$\frac{1}{5}$&$\frac{1}{5}$&$\frac{2}{5}$&$\frac{3}{10}$\\
0 &	1 &	0 &	0 &	$\frac{2}{5}$ &	0 &	0 &	0 &	$\frac{1}{5}$&$\frac{2}{5}$&$\frac{1}{5}$&$\frac{3}{10}$\\
1 &	1 &	0 &$\frac{2}{5}$ &	0 &	0 &	$\frac{1}{5}$&	0 &	$\frac{1}{5}$ &0&0&0\\
0 &	0 &	1 &	$\frac{2}{5}$&	0 &	0 &	0&	0 &	$\frac{2}{5}$&$\frac{1}{5}$ &$\frac{1}{5}$&$\frac{3}{10}$\\
1 &	0 &	1 &	0 &$	\frac{2}{5}$ &0 &	$\frac{1}{5}$ &	0 &	0 &$\frac{1}{5}$&0&0\\
0 &	1 &	1 &	0 &	0 &$	\frac{2}{5} $&	$\frac{1}{5}$ &	0 &	0&0&$\frac{1}{5}$&0\\
1 &	1 &	1 &	0 &	0 &	0 &	0 &	$\frac{2}{5}$ &	0&0 &0&$\frac{1}{10}$\\
		\end{tabular}
	\caption{Extremal pmfs  $d=3, p=\frac{2}{5}$}
	\label{tab:d}
\end{table}

\end{example}

The dimension of the system \eqref{LS} increases very fast because the number of unknown is $D=2^d$.  Therefore finding all the extremal generators of the system becomes computationally infeasible. For example for the moderate-size case $d=6$ and $p=\frac{2}{5}$ there are $1,251,069$ extremal pmfs. Proposition \ref{multinulli} states that the support of an extremal pmf has at most $d+1$ points. One possible approach for finding some extremal pmfs could be based on the selection of $d+1$ components of $\ff$ and set the others equal to zero. This is equivalent to consider a submatrix $H_1$ of $H$ made by $d+1$ columns of $H$.  The extremal points  of  $H_1$ must be computed. The problem is that we are interested in the \emph{positive} solutions of the system $H_1\zz, \, \zz\in \RR^{d+1}_+$. As we show in Example \ref{ex:diff} this approach could be extremely inefficient because many choices of $H_1$ provide problems where positive solutions apart from the trivial $\ff=\boldsymbol{0}$ do not exist.

\begin{example} \label{ex:diff}
We consider the class $\FFF_5(2/5)$, i.e.  the case  $d=5$ and $p=\frac{2}{5}$. In this case the matrix $H$ has $d=5$ rows and $D=32$ columns. We can compute all the extremal pmfs of $H$. There are $5,162$ extremal pmfs. For $1,000$ times we repeat the random selection of the submatrix $H_1$ of $H$ ($H_1$ is made by $d+1=6$ columns of $H$) and the computation of the extremal pmfs of $H_1$ . We find a non-empty set of extremal pmfs only in $162$ cases, providing an estimate of the success of this simple approach equal to $16.2\%$.
\end{example}

Given the computational limitation of the above representation, we propose a different approach.
The next section describes an algebraic representation of the class $\FFF_d(p)$ that reduces the complexity of the problem and provides a new analytical class of generators. The new generators are themselves extremal, very simple and allow us to easily generate pmfs in the class.

\section{Main results}\label{MainR}
 We make the non restrictive hypothesis that $p\leq 1/2$, i.e. $s\leq t$. Since  $p\in \QQ$, let  $p=s/t$. We get  $c={q}/{p}={(t-s)}/{s}$.
Let $\xx=(x_1, \ldots, x_{d-1})$ and $\xx_i=(1 ,  x_i)^T$. Let us consider the row vectors $\mm_+(\xx)=(m_1(\xx),\ldots, m_{D/2}(\xx))^T:=\xx_{d-1}\otimes\ldots\otimes \xx_1$ and $\mm_-(\xx)=(-m_{D/2}(\xx)+\frac{2s-t}{s},\ldots, -m_1(\xx)+\frac{2s-t}{s})^T$ and finally let $\mm(\xx)=(\mm_+(\xx)||\mm_-(\xx))$, the row vector obtained concatenating $\mm_+(\xx)$ and  $\mm_-(\xx)$.

We define the map $\mathcal{H}$ from $\FFF_d(p)$ to the polynomial ring with rational coefficients $\mathbb{Q}[x_1,\ldots, x_d]$ as:
\begin{equation}\label{Eh}
\begin{split}
&\mathcal{H}: \FFF_d(p)\rightarrow \mathbb{\QQ}[x_1,\ldots, x_{d-1}]\\
&\mathcal{H}:\ff\rightarrow P_f(\xx)=\mm(\xx) \ff.
\end{split}
\end{equation}
We call $\mathcal{C_H}$ the image of $\FFF_d(p)$ through $H$.

By construction, a rearrangement of the coefficient of $P_f(\xx)$ gives

\begin{equation}\label{pol}
P_f(\xx)=\mm(\xx)\ff=\sum_{\alpha\in\{0,1\}^d}a_{\alpha}x^{\alpha}, 
\end{equation}
where $ a_{\alpha}\in \QQ$ are linear combinations of the elements of $\ff$. Specifically, it holds
\begin{equation*}\label{matrixQ}
\boldsymbol{a}=Q\ff,
\end{equation*}
where $\aa=(a_{\alpha}:  \alpha\in\{0,1\}^{d-1})$,
\begin{equation}\label{Eq}
Q=(I(2^{d-1}) ||\tilde{I}(2^{d-1})+\tilde{A}),
\end{equation}
$I(2^{d-1})$ is the $2^{d-1}$ dimensional identity matrix, $\tilde{I}(2^{d-1})$ is the $2^{d-1}$ dimensional matrix with $-1$ on the secondary diagonal and all other entries equal to zero and
 $\tilde{A}$ is a $2^{d-1}$ square matrix whose entries are $\tilde{a}_{1j}=(2s-t)/s, \, j=1,\ldots 2^{d-1}$ and $\tilde{a}_{ij}=0$, $j=2,\ldots, 2^{d-1}$.
 \begin{example}

As in Example \ref{example} we consider $\FFF_3(2/5)$, thus $d=3$, $s=2$, $t=5$, $p=\frac{2}{5}$, $c=\frac{3}{2}$ and $\frac{2s-t}{s}=-\frac{1}{2}$. Given $\ff\in \FFF_3(2/5)$, we have $\boldsymbol{a}=Q\ff$, where

\begin{equation*}
Q=\left(\begin{tabular}{r r  r  r  r r r r } \label{tabray}
1&0&0&0&$-\frac{1}{2}$&$-\frac{1}{2}$&$-\frac{1}{2}$&$-1-\frac{1}{2}$	 \\
0&1 &0 &	0&0&0&-1&0\\
0& 0&1 &	0&0&-1&0&0\\
0& 0&0 &	1&-1&0&0&0\\
\end{tabular}\right).
\end{equation*}
For example the image of $\rr_9$ (one of the extremal pmfs listed in Table \ref{tab:d}) is $P_{\rr_9}(\xx)=-0.3(1-x_1-x_2+x_1x_2)$.

In the view of the proof of Theorem \ref{map}, we observe that the elements of  $\mm_+(\xx)=\mm_+(x_1,x_2)=(1,x_1)\otimes (1,x_2)=(1 \,\,\,x_1 \,\,\,x_2\,\,\, x_1x_2)$ computed in $\xx=(-3/2,1)$, $\xx=(1, -3/2)$ and $\xx=(1,1)$  can be put in a one-to-one correspondence with the first $D/2=4$ columns of the matrix $H$ of Equation \eqref{H3} as represented in Table  \ref{AA},
where $H[i,\{1,2,3,4\}]$ is the $i$-th row of the matrix containing the columns $\{1,2,3,4\}$ of $H$, $i=1,2,3$.
\begin{table}\caption{}\label{AA}
\begin{equation*}
\begin{array}{rcccccc}
\mm_+(\xx)&=& 1 & x_1& x_2 & x_1x_2&  \\
\hline
\mm_+(-3/2, 1)&=&1&-\frac{3}{2}  &1&-\frac{3}{2}& H[1,\{1,2,3,4\}] \\
\mm_+(1,-3/2)&=&1 & 1&-\frac{3}{2} & -\frac{3}{2}&H[2,\{1,2,3,4\}]\\
\mm_+(1, 1)&=&1 & 1& 1 & 1& H[3,\{1,2,3,4\}]
\end{array}%
\end{equation*}
\end{table}
Similarly we  observe that the elements of $\mm_-(\xx)=\mm_-(x_1,x_2)=(-x_1x_1-1/2 \,\,\,-x_2-1/2 \,\,\,-x_1-1/2\,\,\,-1-1/2)$ computed in $\xx=(-3/2,1)$, $\xx=(1, -3/2)$ and $\xx=(1,1)$    can be put in a one-to-one correspondence with the last  $D/2=4$ columns of the matrix $H$ of Equation \eqref{H3} as reported in Table \ref{BBc},
where $H[i,\{5,6,7,8\}]$ is the $i$-th row of the matrix containing the columns $\{5,6,7,8\}$ of $H$, $i=1,2,3$.
\begin{table}\caption{}
\begin{equation*}
\begin{array}{rcccccc}
\mm_{-}(\xx)&=&- x_1x_2-\frac{1}{2} & -x_2-\frac{1}{2}& -x_1-\frac{1}{2}&-\frac{3}{2}&\\
\hline
\mm_-(-3/2, 1)&=&1&-\frac{3}{2}  &1&-\frac{3}{2}&H[1,\{5,6,7,8\}]  \\
\mm_-(1,-3/2)&=&1 & 1&-\frac{3}{2} & -\frac{3}{2}&H[2,\{5,6,7,8\}]\\
\mm_-(1, 1)&=&-\frac{3}{2} & -\frac{3}{2}&-\frac{3}{2} & -\frac{3}{2}& H[3,\{5,6,7,8\}]
\end{array}%
\end{equation*}
\end{table}\label{BBc}

%

\end{example}

\begin{theorem}\label{map}
The map $\mathcal{H}$ maps $\FFF_d(p)$ into polynomials $\sum_{\alpha}a_{\alpha}\xx^{\alpha}\in \QQ[x_1, . . . , x_{d-1}]$, $\alpha\in\{0,1\}^{d-1}$ of  the ideal $\II$ of  polynomials   which vanish at points $\mathcal{P}=\{\boldsymbol{1}_{d-1}, \cc_j, j=1,\ldots,d-1 \}$, where $\boldsymbol{1}_{d-1}=(1,\ldots, 1)$ and $\cc_j=(1,\ldots,1, \underbrace{-c}_\text{j-th element},1, \ldots1)$.

\end{theorem}
\begin{proof}
Let $H(d)$, $\mm_+^d(\xx)$ and $\mm_+^d(\xx)$  be the matrices $H$ associated to $\FFF_d(p)$, $\mm_+(\xx)$ and $\mm_-(\xx)$ in dimension $d$, respectively.
Let $\ff\in \FFF_d(p)$ and $H_{j.}(d)$ be the $j$-th row of $H(d)$. We show that $H_{j.}(d)\ff$ is the polynomial $P_f(\xx)$ evaluated at $\xx=\cc_j$, that is $H_{j.}(d)\ff=P_f(\cc_j)$, $j=1,\ldots, d-1$, and  $H_{d.}(d)\ff=P_f(\boldsymbol{1}_{d-1}).$

Let $H^+(d)$ be the submatrix of the first   $2^{d-1}$ columns  of $H(d)$ and $H^-(d)$ the submatrix of the last $2^{d-1}$ columns of $H(d)$, i.e. $H(d)=(H^+(d)|| H^-(d))$.

We consider the following two steps. Step one
proves  by induction that $H^+_{j.}(d)=\mm^d_+(\cc_{j})$ , for $j=1, \ldots, d-1$  and  $H^+_{d.}(d)=\mm^d_+(\boldsymbol{1}_{d-1})$.
We consider  $d=2$.  It holds $\cc_1=-c$, $\boldsymbol{1}_{d-1}=1$ and $\mm^2_+(\xx)=(1, x)^T$. We have   $\mm^2_+(\cc_1)=(1, -c)^T$  and $\mm^2_+(\boldsymbol{1}_1)=(1, 1)^T$.
Since
$H^+_{1.}=(1,-c )^T$ and $H^+_{2.}=(1, 1)^T$, the case $d=2$ is proved.

Let us assume that the assert  is true for $d$.  We prove it for $d+1$. We consider separately three cases: $j=1,\ldots {d-1}$, $j=d$ and $j=d+1$.
For $j=1,\ldots {d-1}$, we have $x_j=-c$ and then it must be $x_d=1$ by construction.
Therefore,
$$H^+_{j.}(d+1)=(1, 1)^T\otimes H^+_{j.}(d)=(1,1)\otimes\mm^d_+(\cc_j)=(\mm^d_+(\cc_j)||1)=\mm_+^{d+1}(\cc_j).$$

For $j=d$ we have $x_d=-c$ and
$$H^+_{d.}(d+1)=(1, -c)^T\otimes H^+_{d}(d)=(1, -c)\otimes \mm_+^d(\boldsymbol{1}_d)=\mm_+^{d+1}(\boldsymbol{c}_d).$$

Finally, for $j=d+1$ we have $$H^+_{d+1 .}(d+1)=(1, 1)^T\otimes H^+_{d.}(d)=(1, 1)\otimes \mm_+^d(\boldsymbol{1}_d)=\mm_+^{d+1}(\boldsymbol{1}_d)$$
and the step one is proved.
Step two proves that $H^-_{j.}(d)=\mm^d_-(\cc_{j})$, for $j=1, \ldots, d-1$ and  $H^-_{d.}(d)=\mm^d_-(\boldsymbol{1}_{d-1})$. It is sufficient to  observe that $\mm^d_-(\xx)=(-m_{D/2}(\xx)+\frac{2s-t}{s},\ldots -m_1(\xx)+\frac{2s-t}{s})^T$ and $H^-=(-H^+_{1.}+\frac{2s-t}{s}, \ldots, -H^+_{d.}+\frac{2s-t}{s})$ and the assert of step two easily follows.

As a consequence $H_j(d)=\mm^d(\cc_j),  \, j=1,\ldots, d-1$ and $H_d(d)=\mm^d(\boldsymbol{1}_{d-1})$. Since $\ff\in \FFF_d(p)$ iff $H(d)\ff=0$ we have $\mm^d(\cc_j)\ff=0$ and $\mm^d(\boldsymbol{1}_{d-1})\ff=0$. Therefore $P_f(\xx)=\mm^d(\xx)\ff\in \II$ and by construction (see \eqref{pol}) $\mm^d(\xx)\ff=\sum_{\alpha}a_{\alpha}\xx^{\alpha}$, $\alpha\in\{0,1\}^{d-1}$, with $\aa=Q\ff$ .

\end{proof}

\begin{remark}
The map $\mathcal{H}$ is not injective.
In Example \ref{example}, $\{\rr_1, \rr_2,\rr_3,\rr_4,\rr_5\}\subseteq Ker(\mathcal{H})$.
\end{remark}
It follows that given a polynomial $P(\xx)=\sum_{\alpha}a_{\alpha}\xx^{\alpha}\in \II$, $\alpha\in\{0,1\}^{d-1}$ there are many pmfs in $\mathcal{H}^{-1}(P(\xx))$ as we study in the remaining part of this section.

The following three steps provide an algorithm to find one pmf $\ff_P=(f_1,\ldots, f_D)\in \FFF_d(p)$ associated to a given $P(\xx)=\sum_{\alpha\in\{0,1\}^{d-1}}a_{\alpha}\xx^{\alpha}\in \mathcal{I}_{\mathcal{P}}$.
\begin{enumerate}\label{AlgoPol}
\item
For each   $j\in\{2,\ldots, D/2\}$, if  $a_{j}\geq0$, then $f_j=a_{j}$ and $f_{D+1-j}=0$;  if $a_j<0$ then $f_{j}=0$ and $f_{D+1-j}=-a_j$.
\item
Let $c_0=\sum_{j=1\\,
a_j<0}^{d/2}a_j\frac{2s-t}{s}+a_1$.Then if $c_0\geq0$,  $f_1=c_0$ and $f_D=0$, if  $c_0<0$ then $f_1=0$ and $f_D=-\frac{c_0}{c}$, where $c=\frac{q}{p}$;

%
%
%
%
%
%
%
\item normalize $\ff_p$ getting, with a small abuse of notation $\ff_p:=\ff_p/(\sum_{j=1}^Df_j)$.
\end{enumerate}

It is easy to verify that $\ff_p\in \FFF_d(p)$ and that $\mathcal{H}(\ff_p)=P(\xx)$.
We call $\ff_p$ the type-0 pmf associated to the polynomial $P$.
We describe this algorithm in Example \ref{descr:alg}.
\begin{example}\label{descr:alg}
We consider two polynomials: the first one has a positive constant term, the second one a negative  constant term.
Consider the case of Example \ref{example} $(d=3, p=2/5)$ and let $P(\xx)$ be the polynomial $P(\xx)=x_1x_2-x_1-x_2+1\in \MC$.
We have $a_1=1$, $a_3=a_3=-1$, $a_4=1$ and then:
\begin{enumerate}
\item step 1  immediately gives $\ff_P=(f_1,0,0,1,0,1,1,f_8)$;
\item from step 2 we get $c_0=\frac{1}{2}+\frac{1}{2}+1=2$, thus $f_1=c_0=2$, $f_8=0$ and $\ff_P=(2,0,0,1,0,1,1, 0)$;
\item the normalization step gives  $\ff_P=(\frac{2}{5},0,0,\frac{1}{5},0, \frac{1}{5},\frac{1}{5},0)$. We observe that $\ff_P=\rr_4$, see Table \ref{tab:d}.
\end{enumerate}
Consider now the  polynomial $P(\xx)=-x_1x_2+x_1+x_2-1\in \MC$. We have $a_1=-1$, $a_2=a_3=1$, $a_4=-1$ and then:
\begin{enumerate}

\item step 1 immediately gives  $\ff_P=(f_1,1,1,0,1,0,0,f_8)$;
\item from  step 2 we get $c_0=-1/2$, thus $f_1=0$, $f_D=\frac{1}{2}\cdot \frac{2}{3}=\frac{1}{3}$ and  $\ff_P=(0,1,1,0,1,0,0,\frac{1}{3})$;

\item the normalization step gives  $\ff_P=(0,\frac{3}{10},\frac{3}{10},0, \frac{3}{10},0,0,\frac{1}{10})$. We observe that $\ff_P=\rr_9$, see Table \ref{tab:d}.

\end{enumerate}
\end{example}
In Proposition \ref{ker} given a polynomial $P(\xx)\in \mathcal{I_P}$ we determine all the pmfs such that $\mathcal{H}(f)=P(\xx)$.
\begin{proposition}\label{ker}
 Given a polynomial $P(\xx)=\sum_{\alpha}a_{\alpha}\xx^{\alpha}\in \mathcal{I_P}$, $\alpha\in\{0,1\}^{d-1}$, $$\mathcal{H}^{-1}(P(\xx))=\{\ff_p+\ee_k, \,\,\ \ee_k\in Ker (\mathcal{H})\},
$$
where $\ff_p$ is the type-0 pmf associated to $P(\xx)$. A basis of $ Ker (\mathcal{H})$ is
\begin{equation*}
\begin{split}
\mathcal{B}_K&=\{(q,0,\ldots, 0,p); (1-2p,p,0,\ldots,0,p,0); (1-2p,0,p,0,\ldots0,,p,0,0);\\
&\ldots (1-2p,0,\ldots, 0,p,p,0\ldots,0)\}.
\end{split}
\end{equation*}
\end{proposition}
\begin{proof}
Notice that  $Ker(\mathcal{H})$ in \eqref{Eh} coincides with $Ker(Q)$, where $Q$ is the matrix  given in \eqref{Eq} which  is the coefficient matrix of the linear application $\mathcal{H}$ between $\RR^{D}$ and $\RR^{D/2}$. Since $\rank(Q)=D/2$ (it is enough to observe that the first $D/2$ columns of $Q$ are the identity matrix), we have $\rank(Ker(Q))=D/2$. Now it sufficient to observe that $\mathcal{B}_K$ is a set of  $D/2$  linearly independent vectors.

\end{proof}
By construction we observe that type-0 pmfs are characterized by the condition that only one of  the components $f_i$ or $f_{D-i+1}$ of $\ff_p$  can be different from zero.
We can now classify the pmfs in $\FFF_d(p)$ as follows.

\begin{definition}
We say that $\ff=(f_1,\ldots, f_D)$ is of type-0 if it is the particular solution $\ff_P$ corresponding to a polynomial $P(\xx)\in \mathcal{I_P}$, it is of  type-1K if $\ff\in Ker(\mathcal{H})$  and it is of type-1 otherwise.
\end{definition}
\begin{example}
In the Example \ref{example} the extremal point $\rr_9=(0,\frac{3}{10},\frac{3}{10},0,\frac{3}{10},0,0,\frac{1}{10})$ is of type-0 and $\rr_5=(\frac{3}{5},0,0,0,0,0,0,\frac{2}{5})$ is of type-1K,  while $\rr_6=(0, \frac{1}{5},\frac{1}{5},\frac{1}{5},\frac{2}{5},0,0,0)$  is of type-1, since $f_4=f_{8-3}=f_5$. Notice that $\mathcal{H}(\rr_6)=1/5(1-x_1+x_2-x_1x_2)$ and $\mathcal{H}(\rr_5)=0$
\end{example}

\begin{proposition}
The Gr\"obner basis of $\II$ with respect to the lexicographic order is $GB=\{G_i(\xx)=(x_i-1)(x_i+c), G_{ik}(\xx)=1-x_i-x_x+x_ix_k,  i,k=,\ldots, 2^{d-1}, i<k\}$. A basis of the quotient space $\mathbb{Q}[x_1,\ldots, x_{d-1}]/\II$ is $\{1,x_1,\ldots, x_{d-1}\}$.
\end{proposition}
\begin{proof}
Let $\II=\langle G_i(\xx)\rangle$.
We prove that  the variety $V(\II)$ of $\II$ is $\mathcal{P}$. $\mathcal{P}\subset V(\II)$ is easy to verify, so we prove the other inclusion.  If $\xx\in V(\II)$, then $x_j=1$ or $x_j=-c$. In fact if there is $k:x_k\neq 1, -c$ then $G_k(\xx)=(x_k-1)(x_k+c)\neq 0$. If $x_k=-c$ then $x_j=1, \forall j\neq k, j=1,\ldots, d-1.$ In fact if there is $i: x_i=-c$, then $G_{ik}(\xx)=x_ix_k-x_1-x_k+1\neq 0$ and $\xx\notin V(\II)$. Thus $\xx=\boldsymbol{1}_{d-1}$ or $\xx=\cc_j, j=1,\ldots d-1$ that is $\xx\in \mathcal{P}$.

$GB$ is a Gr\"oebner basis for $\II$.
Let $A(\xx)\in \II$ and let $a^*x^{\alpha}$ its leading term. If $x_i^2\not\vert  a^*x^{\alpha}$ and $x_ix_j\not\vert a^*x^{\alpha}$, that is $x_ix_j$ does not divide $a^*x^{\alpha}$, for any $i, j$ then $a^*x^{\alpha}=a^*x_k$ for a $k=1,\ldots d-1$. Thus $A(\xx)$ has degree one and cannot be zero both in $\cc_k$ and $\boldsymbol{1}_{d-1}$. Thus, since $A(\xx)\in \II$,  its leading term $LT(A(\xx))$ must be divisible by one of the $LT(G(\xx)), G(\xx)\in GB$ and $GB$ is a Gr\"oebner basis.

\end{proof}

\begin{proposition}
For each $n\in \NN$, $2\leq n\leq d-1$, the monomials $\pi_{j_1,\ldots, j_n}(\xx)=\prod_{i=1}^nx_{j_i}$ have remainder $R_{j_1,\ldots, j_n}(\xx)=\sum_{i=1}^nx_{j_i}-(n-1)$.

\end{proposition}

\begin{proof}
Without loss of generality we consider the monomials  $\pi_n(\xx):=\pi_{1,\ldots, n}(\xx)=x_1\cdots x_n$ and $R_n(\xx):=R_{1,\ldots, n}(\xx)$.

By induction. If $n=2$ this follows because $G_{ij}(\xx)=x_ix_j-x_i-x_j+1\in \II$ and therefore $x_ix_j$ has the remainder $x_i+x_j-1$.

We now prove that $n\rightarrow n+1$. By inductive hypothesis
\begin{equation*}
\begin{split}
\prod_{i=1}^{n+1}x_j=x_{n+1}\prod_{j=1}^{n}x_j=x_{n+1}(I(x)+\sum_{j=1}^nx_j-(n-1)),
\end{split}
\end{equation*}
where $I(\xx)=x_1\cdot\ldots\cdot x_n-\sum_{i=1}^nx_i+n-1\in \mathcal{I_P}$. Thus
\begin{equation*}
\begin{split}
\prod_{i=1}^{n+1}x_j&=x_{n+1}(I(\xx)+\sum_{j=1}^nx_j-(n-1))\\
&=x_{n+1}I(x)+\sum_{j=1}^nx_{n+1}x_j-x_{n+1}(n-1)\\
&=\tilde{I}(x)+\sum_{j=1}^n(G_{n+1, j}(\xx)+x_{n+1}+x_j-1)-x_{n+1}(n-1)
\end{split}
\end{equation*}
where $\tilde{I}(\xx)\in I$. It follows that
\begin{equation*}
\begin{split}
\prod_{i=1}^{n+1}x_j
&=\tilde{I}(x)+\sum_{j=1}^nG_{n+1, j}(\xx)+nx_{n+1}+\sum_{j=1}^nx_j-n-nx_{n+1}+x_{n+1}\\
&=\tilde{I}^*(x)+\sum_{j=1}^nx_j-n+x_{n+1}=\tilde{I}^*(x)+\sum_{j=1}^{n+1}x_j-n,
\end{split}
\end{equation*}
where $\tilde{I}^*(\xx)=\tilde{I}(x)+\sum_{j=1}^nG_{n+1, j}(\xx)\in I$, and the assert is proved.
\end{proof}

We call the polynomials $F_{j_1,\ldots, j_n}(\xx)=\pi_{j_1,\ldots, j_n}(\xx)-R_{j_1,\ldots, j_n}(\xx)$ fundamental polynomials of the ideal $I_{\mathcal{P}}$. In particular we denote by $F_n(\xx)$ the fundamental polynomial $F_n(\xx):=F_{1,\ldots, n}(\xx)$, $n=2,\ldots, d-1$.

\begin{corollary}\label{polI}
The polynomials $P_f(\xx)=\mm(\xx)\ff$ are linear combinations of the fundamental polynomials. In particular they have the form $P_f(\xx)=Q(\xx)-R(\xx)$, where
$Q(\xx)=\sum_{k=2,\ldots, d-1, j_1<j_2\ldots<j_{k}}a_{j_1\ldots j_{k}}x_{j_1}\ldots x_{j_{k}}$ and  $R(\xx)=b_0+\sum_{j=1}^{d-1}b_jx_j$ is the remainder.
\end{corollary}
%

\begin{proof}
From
\begin{equation}
P_f(\xx)=\mm(\xx)\ff=a_{\emptyset}+\sum_{k=1}^{d-1}a_{k}x_{k}+\sum_{k=2}^{d-1}\sum_{j_1<j_2\ldots<j_k}a_{j_1\ldots j_k}x_{j_1}\ldots x_{j_k}, \,\,\, a_{\emptyset}, a_{j_1\ldots j_k}\in \QQ,
\end{equation}
Let  $$Q(\xx)=\sum_{k=2}^{d-1}\sum_{j_1<j_2\ldots<j_k}a_{j_1\ldots j_k}x_{j_1}\ldots x_{j_k}.$$
Since $\{1, x_1,\ldots, x_{d-1}\}$ is a basis of the quotient space and $P_f(\xx)\in \II$
\begin{equation}\label{restQ}
-(a_0+\sum_{k=1}^{d-1}a_{k}x_{j_1})
\end{equation}
is the remainder of $Q(\xx)$.

We also have that the reminder of $Q(\xx)$ is $\sum_{k=2}^{d-1}\sum_{j_1<j_2\ldots<j_k}a_{j_1\ldots j_{d-1}}\pi_{j_1,\ldots, j_k}(\xx).$ 
 Since \eqref{restQ}, then $a_0+\sum_{j=1}^{k}a_jx_j=-\sum_{k=2}^{d-1}\sum_{j_1<j_2\ldots<j_k}a_{j_1\ldots j_{d-1}}R_{j_1,\ldots, j_{k}}(\xx)$ and the thesis follows.
\end{proof}
Notice that the polynomials in $\MC$ are the same for each marginal probability  $p$. The probability $p$ defines the points $\mathcal{P}$ of the variety $V(I_{\mathcal{P}})$.

We have proved that  the set of fundamental polynomials is a set of generators of $\mathcal{C}_{\mathcal{H}}$. And therefore we can use them to generate $\FFF_d(p)$. All pmfs can be generated as linear combinations of the fundamental polynomials, computing the corresponding type-0 pmf and eventually adding an element of $Ker(\mathcal{H})$.

We now use fundamental polynomials to address two open issues. The first is the search of the extremal generators of the convex polytope $\FFF_d(p)$ and the second is the study of bounds for the distributions in $\FFF_d(p)$. We provide an algorithm to find extremal points, all in principle and many in practice and we find an analytical solution the second problem under the convex order of sums.

\section{An algorithm to find extremal pmfs}\label{algo}

The first step to construct an algorithm for generating extremal pmfs is to prove that type-0 pmfs associated to fundamental polynomials are extremal themselves.

\begin{proposition}\label{type0}
Let $\ff\in \FFF_d(p)$ be the type-0 pmf associated  to a fundamental polynomial. The pmf $\ff$ is an extremal probability mass function.
\end{proposition}
\begin{proof}
Let $\ff$ such that $\mathcal{H}(\ff)=F_{n}(\xx), 2\leq n \leq d-1$ is a fundamental polynomial,
$F_{n}(\xx)=\prod_{i=1}^{n}x_i-\sum_{i=1}^{n}x_i+(n-1)$. Then $\ff$ has support on $n+2$ points $\sup(\ff)=\{i_1,\ldots,i_{n+2}\} \subset \{1,\ldots,D\}$, corresponding to the monomials $-x_j-\frac{2s-t}{s}, j=1,\ldots, n$, $\prod_{i=1}^{n}x_i$ and the constant term.
Let $I^*\in \mathcal{M}((2^d-(n+2))\times2^d)$ be the submatrix of $I_D$ so that $I^*\ff=0$ (it contains the rows $I_D[i,]$ of $I_D$ such that $f_i=0$) . Let us define the matrix
\begin{equation*}
H/I:=H//I^*=\left[ \begin{split}
&H \\
 &I^*
\end{split}\right].
\end{equation*}
Let us consider the columns of $H/I$.
The pmf $\ff$ has mass on $n+2$ points and the corresponding $n+2$ columns $H/I[,j]$ of $H/I$, that we call $H/I_j, j=i_1,\ldots, i_{n+2}$ have all zeros below the rows of $H$,
\begin{equation*}
H/I_j=\left[ \begin{split}
&H_j \\
 &\boldsymbol{0}
\end{split}\right],
\end{equation*}
where $\boldsymbol{0}$ is the $2^{d}-(n+2)$-vector with all zeros.
Among them, consider the columns corresponding to $-x_j-\frac{2s-t}{s}, j=1,\ldots, n$ and constant term. These are $n+1$ independent columns of $H/I$ because it can be proved that $\sum_{j=1}^n \gamma_j (-x_j-\frac{2s-t}{s})+\gamma_{n+1}=0$ for $x\in \mathcal{P}$ iff $\gamma_1=\ldots=\gamma_{n+1}=0$. 
The column corresponding to $\prod_{i=1}^{n}x_i$ is linearly dependent by these $n+1$ independent columns of $H/I$ because $F_n(\xx)$ belongs to the ideal $\II$, that is $F_n(\xx)=0, \xx \in \mathcal{P}$. It follows that $\prod_{i=1}^{n}x_i = \sum_{i=1}^n x_i -n+1$ and then
 $\prod_{i=1}^{n}x_i =- \sum_{i=1}^n (-x_i-\frac{2s-t}{s})-n \frac{2s-t}{s}-n+1$.
The $2^{d}-(n+2)$ remaining columns $H/I_j, j\in\{1,\ldots 2^d\}, j\neq i_1,\ldots, i_{n+2}$  of $H/I$ are independent because $I^*[,j\in\{1,\ldots 2^d\}, j\neq i_1,\ldots, i_{n+2}]$ is an identity matrix. Therefore we have $2^{d}-(n+2)+n+1=2^{d}-1$ independent columns. It follows that
\begin{equation*}
\text{rank}(H/I)=2^{d}-1.
\end{equation*}
From  Lemma 2.3 in \cite{terzer2009large}, $\ff$ is an extremal point.
\end{proof}
\begin{example}
Let us consider $d=3$, $p=2/5$ and $F_2(x_1,x_2)=x_1x_2-x_1-x_2+1$.  As shown in Example \ref{descr:alg}, the corresponding pmf is $\ff=(\frac{2}{5},0,0,\frac{1}{5},0,\frac{1}{5},\frac{1}{5},0)$ and is an extremal pmf. The $H/I$ matrix is
\begin{equation*}
H/I=\left(
\begin{array}{cccccccccccccccc}
1&-\frac{3}{2}  &1&-\frac{3}{2}   & 1&-\frac{3}{2}  &1&-\frac{3}{2} \\
1 & 1&-\frac{3}{2} & -\frac{3}{2}& 1 & 1&-\frac{3}{2} & -\frac{3}{2}\\
 1 & 1& 1 & 1 & -\frac{3}{2} & -\frac{3}{2}& -\frac{3}{2}& -\frac{3}{2}\\
0 & 1 & 0 & 0 & 0 & 0 & 0 & 0 \\
0 & 0 & 1 & 0 & 0 & 0 & 0 & 0 \\
0 & 0 & 0 & 0 & 1 & 0 & 0 & 0 \\
0 & 0 & 0 & 0 & 0 & 0 & 0 & 1
\end{array}%
\right).\label{HI3}
\end{equation*}%
and $\text{rank}(H/I)=2^{3}-1=7$.
\end{example}
%
%
Analogously, it can be proved the following.
\begin{corollary}
If $\mathcal{H}(\ff)=-F(\xx)$, where $F(\xx)$  is a fundamental polynomial, then $\ff$ is an extremal mass function.
\end{corollary}

\begin{remark}
The extremal points associated to $F(\xx)$ and $-F(\xx)$ are different and they are not symmetric. We can see from Example \ref{example} that $\mathcal{H}(\rr_4)=x_1 x_2-x_1-x_2+1$ and $\mathcal{H}(\rr_9)=-x_1x_2+x_1+x_2-1$. This is a consequence of the role played by the constant terms arising from the monomials with negative coefficients, i.e. the monomials in $\mm_{-}(\xx)$.
\end{remark}

We  proved that to  find a point  in $\FFF_d(p)$ it is sufficient to pick a polynomial in $\MC$. From Corollary \ref{polI} we have that any  $P_f(\xx)\in\MC$ can be obtained as a linear combination of fundamental polynomials.
%
Proposition \ref{multinulli} gives a necessary condition for a point to be an extremal point, it must have support on at most $d+1$ points. Therefore to find an extremal  point not associated to a fundamental polynomial we have to choose a linear combination of at most $d+1$ fundamental polynomials, find a corresponding pmf in $\FFF_d(p)$ and finally check if it is an extremal point by means of Lemma 2.3 in \cite{terzer2009large}.

Proposition \ref{ker} provides a way to easily build extremal points of type-1K as $\ff+\boldsymbol{e}$ where $\ff$ is a type-0 pmf and $\boldsymbol{e} \in \ker(\mathcal{H})$. It is enough to choose $\boldsymbol{e}$ in a way that all the components of $\ff+\boldsymbol{e}$ are positive and then the corresponding type-1K  pmf will be obtained normalizing $\ff+\boldsymbol{e}$.
To find type-0 extremal points not associated to the fundamental polynomials is more challenging. We show an algorithm to find type-0 extremal points which requires the construction of a matrix whose columns are the coefficients of the remainders of the monomials of degree greater than one, $x_{j_1}\ldots x_{j_k}$, with $k \geq 2$.

Let $\boldsymbol{\pi}=(x_1x_2,\ldots, x_1\ldots x_{d-1})^T$ be the row vector of the $2^{d-1}-d$ monomials of degree greater than one  ordered according to the reverse-lexicographical criterion. Let $B$ be the matrix whose elements of the  $j$-th column are the coefficients of the remainders of $\boldsymbol{\pi}_j$ corresponding to the basis $\{1,x_1,\ldots, x_{d-1}\}$ of the quotient space, as illustrated in Table \ref{cicli}.

We look for a (column) vector
$\aa=(a_{12},\ldots, a_{12,\ldots d-1})$,  so that  $$P(\xx)=\sum_{k=2}^{d-1} \sum_{j_1<\ldots<j_k} a_{j_1\ldots j_k} F_{j_1\ldots j_k}(\xx)=b_\emptyset + \sum_{k=1}^{d-1} b_k x_k + \sum_{k=2}^{d-1} \sum_{j_1<\ldots<j_k} a_{j_1\ldots j_k} x_{j_1}\ldots x_{j_k}$$ is associated to an extremal point. The term $-(b_\emptyset + \sum_{k=1}^{d-1} b_k x_k)$ is the remainder of $\sum_{k=2}^{d-1} \sum_{j_1<\ldots<j_k} a_{j_1\ldots j_k} x_{j_1}\ldots x_{j_k}$ .
%

By construction, the product $B\aa$ is the column vector of the coefficients of the remainder of $\sum_{k=2}^{d-1} \sum_{j_1<\ldots<j_k} a_{j_1\ldots j_k} x_{j_1}\ldots x_{j_k}$, that is $B\aa=-(b_\emptyset,b_1,\ldots,b_{d-1})$.
If the $k$-th row of $B\aa$ is zero, the $k$-th term of the remainder of $\sum_{k=2}^{d-1} \sum_{j_1<\ldots<j_k} a_{j_1\ldots j_k} x_{j_1}\ldots x_{j_k}$ is zero. The k-th row of $(B\aa)_{k\cdot}$ is $B_{k\cdot}\aa$, where $B_{k\cdot}$ is the $k$-th row of $B$, thus $B_{k\cdot}\aa$ is $-b_{k-1}$, and $B_{1\cdot}\aa$ is $-b_\emptyset$.

The solutions of $B_{k\cdot}\aa=0$ give the coefficients of all the polynomials $P(\xx)$ which are associated to pmfs in $\FFF_d(p)$ without the components corresponding to $x_{k-1}$ and $-x_{k-1}+\frac{2s-t}{s}$. We observe that $b_\emptyset$ is not immediately related to $f_1$ and $f_D$. For example, in the case $d=4$, $p=2/5$, the polynomial $\frac{1}{5}x_1 x_2 + \frac{1}{5}x_1 x_3 +\frac{1}{5}x_2 x_3 - \frac{2}{5}x_1 x_2 x_3 -\frac{1}{5}$ corresponds to the pmf $\ff=(0,0,0,\frac{1}{5},0,\frac{1}{5},\frac{1}{5},0,\frac{2}{5},0,0,0,0,0,0,0)$. We have $b_\emptyset=-1/5\neq0$ and $f_1=f_{16}=0$. Or the polynomial $x_1x_2-x_1x_3+x_2+x_3$  corresponds to the pmf $\ff=(\frac{1}{5},0,0,\frac{1}{5},\frac{1}{5},0,0,0,0,0,\frac{1}{5},0,0,\frac{1}{5},0,0)$. We have  $b_\emptyset=-0$ and $f_1=1/5 \neq 0$.  For this reason we will not consider the equation $B_{1\cdot}\aa=0$.
Let us suppose that we are interested in polynomials $P(\xx)$ where only some $a_{j_1\ldots j_k}$ can be different from zero. We define $J$ as the corresponding set of indexes, ${j_1\ldots j_k} \in J \leftrightarrow a_{j_1\ldots j_k} \neq0$ . Let $B_{\cdot J}$ be  the matrix whose columns are $B_{\cdot j}$, $j\in J$ and $\aa_{J}$ the sub-vector of $\aa$ whose elements are $a_j, j\in J$.
The  elements of the kernel of the linear application
\begin{equation*}\label{LS1}
(B_{\cdot J}\aa_J)_{K\cdot},
\end{equation*}
where $(B_{\cdot J}\aa_J)_{K\cdot}$ are the rows $(B_{\cdot J}\aa_J)_{k\cdot}$ of $(B_{\cdot J}\aa_J)$, $k\in K\subseteq \{2,\ldots, d\}$,
are the coefficients $\aa$ such that the  polynomial $P(\xx)$ does not have the terms $x_{k-1}, \, k\in K\subseteq\{2,\ldots, d\}$.

The polynomials associated to extremal points, $P(\xx)$,  must have at most $d+2$ non-zero coefficient ($d+2$ because the constant term $b_\emptyset$ does not always correspond to have $f_1$ or $f_D$ in the support of the corresponding pmf). Therefore, we choose $\#J$ monomials in $\boldsymbol{\pi}$, with $\#J\leq d+2$, that means $\#J$ columns of the matrix in Table \ref{cicli}. Then we look for the polynomials  $P(\xx)$ so that the remainder has $1,\ldots, d-1$ terms equal zero.

Formally we have the following steps:
\begin{enumerate}
\item choose $J$ and $K$ with $\#J+ d-\#K \leq d+2$;
\item find the kernel generators of
\begin{equation*}
(B_{J\cdot}\aa^T)_{K\cdot};
\end{equation*}
\item for each kernel generator $\aa^{(i)}$ consider the corresponding polynomial $P^{(i)}(\xx)$;
\item construct the corresponding pmf in $\FFF_d(p)$;
\item check if it is an extremal point of the polytope.
\end{enumerate}
The above algorithm potentially finds all the type-0 extremal points, in practice we can easily  find many of them because the number of systems to be solved increases with the dimension $d$.

\begin{table}[h!]
\begin{center}\caption{Matrix representation of the remainders of the monomials $x_{j1}\ldots x_{jk}$ }
\begin{tabular}{r |c } \label{cicli}
&$\boldsymbol{\pi}(x)$\\
\hline
$\boldsymbol{R}(x)$&B\\
\end{tabular}\,=\,\begin{tabular}{r |r  r  r  r } \label{tabray}
&$x_1x_2$& $x_1x_3$&$\cdots$&$x_1\cdots x_{d-1}$\\
\hline
1&-1 &	-1&	\ldots &	$-(d-1)$ \\
$x_1$&1 &	1 &	\ldots &1\\
$x_2$&1 &	0 &\ldots &	1 \\
$\ldots$&\ldots&&&\\
$x_{d-1}$&0 &	0 &	0 &	1 \\
\end{tabular}
\end{center}
\end{table}

To the only purpose of illustrating the procedure, in the following example we look for the type-0 extremal points for $d=4$, because the case $d=3$ only has two fundamental polynomial (opposite signs) and it is trivial.

\begin{example}
Consider $\FFF_4(2/5)$, $d=4$ and $p=2/5$.

\begin{table}[h!]
\begin{center}\caption{Matrix representation the remainders of the monomials $x_1x_2, \ldots x_1x_2x_3$ }
\begin{tabular}{r |c } \label{cicli3}
&$\boldsymbol{\pi}(x)$\\
\hline
$\boldsymbol{R}(x)$&B\\
\end{tabular}\,=\,\begin{tabular}{r |r  r  r  r } \label{tabray}
&$x_1x_2$& $x_1x_3$&$x_2x_3$&$x_1x_2x_3$\\
\hline
1&-1 &	-1&	-1 &	$-2$ \\
$x_1$&1 &	1 &	0&1\\
$x_2$&1 &	0 &1 &	1 \\
$x_3$&0 &	1&	1 &	1 \\
\end{tabular}
\end{center}
\end{table}
The extremal pmfs have support on at most $d+1=5$ points. Since the remainder has 4 terms, if we decide to combine two fundamental polynomials we have to eliminate at least one monomial of the remainder. By so doing, the remainder will have at most $3$ terms  and  the polynomial will have at most 5 coefficients. As an example we choose the first two columns, i.e.  $J=\{1,2\}$. We have
$\aa=(a_{12}, a_{13})$ and
\begin{equation*}
B_{\cdot J}=\left(\begin{tabular}{rr} \label{BJ}
-1 &-1 \\
1 &1\\
1 &0\\
0 &1 \\
\end{tabular}
\right)
\end{equation*}
We have
\begin{equation*}
B_{\cdot J}\aa=\left(\begin{tabular}{rr} \label{BJ2}
-1 &-1 \\
1 &1\\
1 &0\\
0 &1 \\
\end{tabular}
\right)\aa=\left(\begin{tabular}{r}
$-a_{12}-a_{23} $\\
$a_{12}+a_{23} $\\
$a_{12}$\\
$a_{23}$\\
\end{tabular}
\right)
\end{equation*}

We look for the polynomial without the term corresponding to $x_1$, that is $K={2}$. These are the solutions of the homogeneous system:

\begin{equation*}
(B_{\cdot J}A^T)_2=a_{12}+a_{23}=0
\end{equation*}

We can choose $\aa=(1, -1)$ and we obtain $P(x)=x_1x_2-x_1x_3+x_2-x_3$. The associated pmf is $(\frac{1}{5},0,0,\frac{1}{5},\frac{1}{5},0,0,0,0,0,\frac{1}{5},0,0,\frac{1}{5},0,0)$ and we can verify that is an extremal point.
\end{example}

In the next section we use fundamental polynomials to find a pmf in $\FFF_d(p)$ that satisfies a given condition. We show that it is convenient to  set the condition in terms of coefficient of a polynomial in $\mathcal{C}_H$ and find a corresponding pmf.

\section{Lower bounds for the  convex order} \label{lcx}

A problem extensively addressed in statistics and applied probability is to find  upper and lower bounds for sums $S=X_1+\ldots+X_d$ of random variables $X_i$ of which the marginal distributions are known but the joint distribution is unspecified \cite{kaas2000upper}. These bounds are linked to the highest and lowest dependence structure and they are  actually bounds in the sense of the convex order.  We recall the definition of  the convex order.

\begin{definition}
Given two random variables $X$ and $Y$ with finite means, $X$ is said to be smaller
than $Y$ in the convex order (denoted $X\le_{cx}Y$) if
$$
E[\phi(X)]\leq E[\phi(Y)]
$$
for all real-valued convex functions $\phi$ for which the expectations exist.
\end{definition}
 The convex
order is a variability order, in fact it  is easy to verify that
$X\le_{cx}Y$ implies $E[X]=E[Y]$, and
${V}[X]\leq{{V}[Y]}$.
It can also be proved, see e.g. \cite{shaked2007stochastic}, that
\begin{center}
$X\leq_{cx}Y$ iff $E[X]=E[Y]$ and $E[(X-l)^+]\leq E[(Y-l)^+]$ for all $l\in \RR_+$,
\end{center}
where $x^+=max\{x,0\}$. The last inequality defines the so called stop-loss order, that is important in insurance. See \cite{denuit1998optimal} for a discussion on the relationship between convex and stop-loss orders.

We look for the minimum convex order for sums of   Bernoulli variables with mean $p$, when the joint distribution is unspecified. The problem to find the upper bound is solved and it is well known that the upper bound is the upper Fr\'echet bound of the class, that is the extremal point $\rr^U$ with support on the two points $(0,\ldots,0)$ and $(1,\ldots,1)$. We look for $\XX\in \FFF_d(p)$ such that $S=\sum_{i=1}^dX_i$ is minimal in the sense of the convex order.
Since we consider  $\XX\in \FF_d(p)$ the sums have all the same mean $dp$, thus our problem reduce to find
\begin{equation}\label{minB}
\XX^*=\text{argmin}_{\XX\in \FFF_d(p)}E[(S-l)^+],
\end{equation}
where $S=\sum_{i=1}^dX_i$.

 Let $\mathcal{D}(dp)$ be the class of discrete distributions on $\{0,\ldots, d\}$ with mean $dp$, clearly $S\in \mathcal{D}(dp)$.
The paper \cite{fontana2020model}  proves that the class of sums of exchangeable Bernoulli distributions with the same  mean $p$ coincides
with the entire class of discrete distributions with mean $dp$, $\DDD(dp)$. Therefore the map

\begin{equation*}  \label{map0}
\begin{split}
H: \mathcal{F}_d(p)&\rightarrow \mathcal{D}(dp) \\
\ff&\rightarrow \pp,
\end{split}%
\end{equation*}
where $\pp$ is the pmf of $S=\sum_{i=1}^dX_i$, is onto on $\mathcal{D}(dp).$

Thanks to the above result
we can look for the convex order  bounds in $\mathcal{D}_d(dp)$ to find firstly the bounds for the sums and then the corresponding multivariate Bernoulli distributions. Formally, we look for
\begin{equation}\label{minS}
S^*=\text{argmin}_{S\in \mathcal{D}(dp)}E[(S-l)^+].
\end{equation}

Then, by means of the results in previous sections, we characterize the  points  $\XX\in\FFF_d(p)$ so that $\sum_{i=1}^dX_i=S^*$.

In \cite{fontana2020model} the authors prove that the class $\mathcal{D}_d(dp)$ is a convex polytope and they explicitly found its generators. This result is stated in Proposition \ref{binu}.
\begin{proposition}
\label{binu}
The extremal pmfs of $\mathcal{D}(dp)$, ${s}_{j_1,j_2}$  have support on two points $(j_1, j_2)$  with $j_1=0,1,\ldots, j_1^{M}$, $j_2=j_2^m, j_2^m+1, \ldots, d$, $j_1^M$ is
the largest integer less than $pd$ and $j_2^m$ is the smallest integer
greater than pd.  They are

\begin{equation}  \label{binul}
s_{j_1, j_2}(y)=\left\{
\begin{array}{cc}
\frac{j_2-pd}{j_2-j_1} & y=j_1 \\
\frac{pd-j_1}{j_2-j_1} & y=j_2 \\
0 & \text{otherwise}%
\end{array}
\right..
\end{equation}

If $pd$ is integer the extremal densities contain also
\begin{equation}  \label{onenul}
s_{pd}(y)=\left\{
\begin{array}{cc}
1 & y=pd \\
0 & \text{otherwise}%
\end{array}
\right..
\end{equation}
 If $pd$ is not integer there are $n_p=(j_1^M+1)(d-j_1^M)$ {extremal points}.
If $pd$ is integer there are $n_p=d^2p(1-p)+1$ {extremal points}.

\end{proposition}

For any $\phi$ the extremal values for $E[\phi(S)]$ are reached on the extremal points (see \cite{fontana2021exchangeable}). Thus, the bounds for the convex order are reached on the extremal points.
 In particular, the upper bound of convex order is on the extremal point $s^U=s_{0,d}$. This is a straightforward consequence of the well known fact that the upper bound for  $S=\sum_{i=1}^dX_i$ is the upper Frech\'et bound $\rr^U\in \FFF_d(p)$ and $S^U=\sum_{i=1}^dR_i^U$ with pmf $s_{0,d}$. On the contrary the lower bound is still an open issue. \cite{dhaene1999safest} proved that if $pd<1$ the lower Fr\'echet bound belongs to $\FFF_d(p)$ and corresponds to the lower bound for convex order. Here we generalize their result for each $p$ and $d$.
We first find  the solution of  \eqref{minS}.
\begin{proposition}\label{min:cx}
The solution of equation \eqref{minS}, i.e.
\begin{equation*}
S^*=\text{argmin}_{S\in \mathcal{D}(dp)}E[(S-l)^+].
\end{equation*}
is $S^*=S_{j^M, j^m}$, where $S_{j^M, j^m}$ is the random variable with pmf $s_{j^M,j^m}$.
\end{proposition}
\begin{proof}
To simplify notation let $j^M=m$, then $j^m=m+1$.
Let us first consider the case where $pd$ is not integer.
We observe that $E[(S-l)^+]$ is defined as an expected value of a function of $S \in \mathcal{D}_d(dp)$. Then, for finding the minimum value of $E[(S-l)^+]$, it will be enough to consider the extremal pmfs of $\mathcal{D}_d(dp)$. We will write $E_{s_{i,j}}[(S-l)^+]$ to highlight that the expected value is computed using $s_{i,j}$ as probability mass function:
\[
E_{s_{i,j}}[(S-l)^+]=\sum_{y=0}^d (y-l)^+ s_{i,j}(y)= (i-l)^+ \frac{j-pd}{j-i} +  (j-l)^+ \frac{pd-i}{j-i}.
\]
We consider three different cases:
\begin{enumerate}
\item $0 \leq l \leq m$;
\item $m < l < m+1$;
\item $m+1 \leq l \leq d$.
\end{enumerate}
In the first case ($0 \leq l \leq m$) we partition the set of all the extremal pmfs of  $\mathcal{D}_d(dp)$, namely $\mathcal{R}_{dp}$, into two disjoint classes:
\[
\mathcal{R}_{dp}^{(1)}=\{s_{i,j}: i=0,\ldots,\lfloor l \rfloor, \; j=m+1,\ldots,d\},
\mathcal{R}_{dp}^{(2)}=\{s_{i,j}: i=\lceil l \rceil ,\ldots,m, \; j=m+1,\ldots,d\},
\]
where $\lfloor l \rfloor$ is the largest integer less or equal than $l$ and $\lceil l \rceil$ is the smallest integer larger or equal than $l$.
For any pmf in $\mathcal{R}_{dp}^{(2)}$ we obtain that $E_{s_{i,j}}[(S-l)^+]=dp-l$. For $s_{i,j} \in \mathcal{R}_{dp}^{(1)}$ we get
\[
E_{s_{i,j}}[(S-l)^+]=\frac{pd-i}{j-i}(j-l)
\]
It is easy to verify that $\frac{pd-i}{j-i}(j-l) > pd - l$. It follows that, for $0 \leq l \leq m$, the pmfs in $\mathcal{R}_{dp}^{(2)}$ minimize the expected value and $E_{s_{i,j}}[(S-l)^+]=pd-l$.

In the second case ($m < l < m+1$), we observe that
\[
E_{s_{i,j}}[(S-l)^+]=\frac{pd-i}{j-i}(j-l), i=0,\ldots,m, \; j=m+1, \ldots, d.
\]
Then it is easy to verify that
\[
E_{s_{i,j}}[(S-l)^+] > E_{s_{m,j}}[(S-l)^+], \; i=0,\ldots,m-1, j=m+1,\ldots,d.
\]
and that
\[
E_{s_{m,j}}[(S-l)^+] > E_{s_{m,m+1}}[(S-l)^+], \; j=m+2,\ldots,d.
\]
It follows that, for $m < l < m+1$, the pmf $s_{m,m+1}$ minimizes  the expected value and $E_{s_{m,m+1}}[(S-l)^+]=(pd-m)(m+1-l)$.

In the third case ($m+1 \leq l \leq d$), we observe that $E_{s_{i,j}}[(S-l)^+]=0$ $i=0,\ldots,m$, $j=m+1,\ldots,\lfloor l \rfloor$ and $E_{s_{i,j}}[(S-l)^+]>0$, $i=0,\ldots,m$, $j=\lceil l \rceil ,\ldots,d$. It follows that, for $m+1 \leq l \leq d$, the extremal pmfs $s_{i,j}, i=0,\ldots,m, j=m+1,\ldots,\lfloor l \rfloor$ minimize the expected value and $E_{s_{i,j}}[(S-l)^+]=0$.

We can conclude that for $0 \leq l \leq d$ the pmf that minimizes $E_{s_{i,j}}[(S-l)^+]$ is $s_{m,m+1}$.

Let us now consider the case where $pd$ is integer. It is easy to verify that $s_{pd,pd}$ is the pmf that minimizes the expected value and
\[
E_{s_{pd,pd}}[(S-l)^+]=
\begin{cases}
pd-l &\text{ if } l \leq pd \\
0 &\text{ otherwise}
\end{cases}
.\]
\end{proof}

Therefore the solutions of \eqref{minB}, that is
\begin{equation*}
S^*=\text{argmin}_{S\in \mathcal{D}(dp)}E[(S-l)^+].
\end{equation*}
 are the pmfs of $\XX$ with $\XX$ in $\mathcal{X}^*=\{\XX\in \FFF_d(p): \sum_{i=1}^dX_i=S^*=S_{j^M, j^M}\}$. As usual with a small abuse of notation if $\ff$ is the vector pmf of $\XX\in \chi^*$ we also write $\ff\in \chi^*$.
\begin{proposition}
If $\XX\in \mathcal{X}^*$, then its pmf $\ff$ is such that
\begin{equation*}
\ff=\sum_{i_1}^{n*}\lambda_i\rr_i^*, \,\,\, \lambda_i\neq 0.
\end{equation*}
where $\rr_i^*$ are the extremal points of $\FFF_d(p)$ in $\mathcal{X}^*.$
\end{proposition}
\begin{proof}
$\XX\in \FFF_d(p)$ then
\begin{equation*}
\ff=\sum_{i=1}^{n}\lambda_i\rr_i, \,\,\,  \rr_i=(r_{i1},\ldots, r_{i2^d}).
\end{equation*}
Let $S_{\XX}=\sum_{i=1}^dX_i$, $S_{\RRR_i}=\sum_{k=1}^dR_{ik}$ and $\chi_k=\{\xx\in \design: \sum_{j=1}^{d}x_j= k\}$.
\begin{equation*}
\begin{split}
P(S_{\XX}=k)&=\sum_{j\in \chi_k} f_j=\sum_{j\in \chi_k} \sum_{i_1}^{n}\lambda_ir_{ij}\\
&=\sum_{i=1}^{n}\lambda_i\sum_{j\in \chi_k} r_{ij}=\sum_{i=1}^n\lambda_iP(S_{\RRR_i}=k).
\end{split}
\end{equation*}
We have  $\XX\in \mathcal{X}^*$ iff $P(S_{\XX}=k)=0$ for $k\neq j^M, j^m$, that is iff $P(S_{\RRR_i}=k)=0$ for $k\neq j^M, j^m$ and the thesis follows.
\end{proof}

We therefore look for the multivariate Bernoulli variables  $\XX\in\FFF_d(p)$ such that $P(S_{\XX}=k)=0$ for $k\neq j^M, j^m$, where $S_{\XX}=\sum_{i=1}^dX_i$. This is equivalent to look for the pmfs with support $\chi_{M}\cup \chi_{m}$, where $\chi_k=\{x\in \chi: \sum_{i=1}^dx_i=k\}$. In the following simple example the points we find are extremal points.

\begin{example}
Consider $\FFF_3(2/5)$, as in Example \ref{example}. In this case $pd=1.2\in(1,2]$. Thus $S^*=S_{1,2}$ is the lower bound for convex order in $\mathcal{D}(1.2)$. The variable $S^*$ has support on $(1,2)$, thus the extremal points $\rr\in \chi^*$ have support out of  $(0,0,0)$ and  $(1,1,1)$. Therefore the   extremal points are: $\rr_6, \rr_7, \rr_8$, see Table \ref{tab:d}.
\end{example}

We have proved that minimal convex sums correspond to a family $\chi^*$ of multivariate Bernoulli pmf.

In \cite{fontana2020model} it is proved that there is a one to one map between  $\mathcal{D}_d(dp)$ and the class of exchangeable Bernoulli distributions with mean $p$, $\EEE(p)\subseteq \mathcal{F}_d(p)$. Consequently there is one exchangeable distribution in $\chi^*$.  In \cite{fontana2020model} the authors also proved that the exchangeable pmf $\ff^*$ associated to $S^*$ is the pmf with minimal correlation. Therefore, under exchangeability,  minimal convex sums corresponds to minimal correlation.

The following proposition proves that if $\ff\in \chi^*$ then its mean correlation, i.e. the mean of the correlations $\rho(X_i, X_j)$ of each pair of variables $X_i$ and $X_j$, $i,j=1,\ldots,d$, $i<j$, is constant and equal to the correlation of $\ff^*$. Therefore the extremal points belonging to $\chi^*$ generates the pmf with the lower mean correlation.

\begin{proposition} \label{prop:mom2}
Let $f \in \mathcal{F}_d$ be the pmf of a $d$-dimensional multivariate Bernoulli random variable $\XX=(X_1,\ldots,X_d)$ and $p_S$ the pmf of the sum $S=\sum_{i=1}^dX_i$, with $p_S(k)=p_k=P(S=k), k=0,\ldots,d$. The sum of the second-order crossed moments of $\XX$ can be written as a linear combination of the values $p_k, k=2,\ldots,d$ of the pmf $p_S$ of $S$:
\[
\sum_{1 \leq i<j \leq d} E[X_iX_j]=\sum_{k=2}^d \binom{k}{2} p_k
\]
\end{proposition}
\begin{proof}
Given a positive integer $k$, let $\design_k=\{0, 1\}^k$ the set of all binary vectors $\xx_k=(x_1,\ldots,x_k)$ of dimension $k$. We denote by $|\xx|_0$ the sum of the elements of $\xx_k$, $ |\xx_k|_0=\sum_{i=1}^k x_i$. Given a binary vector $\xx_d=(x_1,\ldots,x_d) \in  \design_{d} $,  let $J_{\xx_d}$ be the set of indices corresponding to its non-zero component (e.g. $J_{\xx_4}=\{1,4\}$ for $\xx_4=(1,0,0,1))$. If  $\xx_d$ is the null vector,  $J_{\xx_d}$ is the empty set,  $J_{(0,\ldots,0)}=\emptyset$. We write $f_ {J_{\xx_d}}$ in place of $f(x_1,\ldots,x_d)$ to simplify the notation (e.g. $f(1,0,0,1) \equiv f_{14}\equiv f_{41}$). We denote by $\mathcal{C}({A,b})$ the set of all the subsets of size $b$ of $A$, $0 \leq b \leq \#A$. We denote by $A_{ij}$ the subset of $A=\{1,\ldots,d\}$ which contains all the elements of $A$ except $i$ and $j$, $1 \leq i < j \leq d$.

From the definition of second-order crossed moments we get
\begin{eqnarray*}
E[X_1 X_2]=\sum_{(x_1,\ldots,x_d) \in \design_d} x_1x_2f(x_1,\ldots,x_d) = \sum_{(x_3,\ldots,x_d) \in \design_{d-2}} f(1,1,x_3\ldots,x_d) = \\
=\sum_{r=0}^{d-2}  \sum_{\substack{\xx_{d-2}=(x_3,\ldots,x_d) \in \design_{d-2} \\ |\xx_{d-2}|_0=r}} f(1,1,x_3\ldots,x_d) =\\
=\sum_{r=0}^{d-2}  \sum_{c\in\mathcal{C}(A_{12},r)} f_{12c}
\end{eqnarray*}
Then we obtain
\[
\sum_{1 \leq i<j \leq d} E[X_iX_j]=\sum_{1 \leq i<j \leq d} \sum_{r=0}^{d-2}  \sum_{c\in\mathcal{C}(A_{ij},r)} f_{ijc} = \sum_{r=0}^{d-2}   \sum_{1 \leq i<j \leq d}\sum_{c\in\mathcal{C}(A_{ij},r)} f_{ijc}
\]
We consider the term
\begin{equation} \label{insum}
\sum_{1 \leq i<j \leq d}\sum_{c\in\mathcal{C}(A_{ij},r)} f_{ijc}
\end{equation}
and observe that all the terms $f_{ijc}$ of the summation involve values of the pmf which are computed on vector $\xx_d$ with $r+2$ ones (and the remaining zeros). We also observe that, by symmetry, each term $f_{ijc}$ must appear in the summation the same number of times of any other, let's say $f_{i'j'c'}$. It follows that
\[
\sum_{1 \leq i<j \leq d}\sum_{c\in\mathcal{C}(A_{ij},r)} f_{ijc} = \alpha_{r+2} p_{r+2}
\]
where $\alpha_{r+2}$ can be determined as the ratio between the number of all the terms in the summation of Equations \ref{insum} and $\binom{d}{r+2}$, that is the number of values of $f$ which are used to compute $p_{r+2}$.
We obtain
\[
 \alpha_{r+2}=\frac{\binom{d}{2} \binom{d-2}{r}}{\binom{d}{r+2}}=\binom{r+2}{2}
\]
It follows
\[
\sum_{1 \leq i<j \leq d} E[X_iX_j]= \sum_{r=0}^{d-2} \binom{r+2}{2}p_{r+2}= \sum_{k=2}^{d} \binom{k}{2}p_{k}
\]
\end{proof}
Proposition \ref{prop:mom2} can be generalized to $\tau$-order crossed moments, for $\tau=2,\ldots,d$ as stated in Proposition \ref{prop:momtau}.

\begin{proposition} \label{prop:momtau}
Let $f \in \mathcal{F}_d$ be the pmf of a $d$-dimensional multivariate Bernoulli random variable $\XX=(X_1,\ldots,X_d)$ and $p_S$ the pmf of the sum $S=\sum_{i=1}^dX_i$, with $p_S(k)=p_k=P(S=k), k=0,\ldots,d$. The sum of the $\tau$-order crossed moments of $\XX$, $\tau\geq 2$, can be written as a linear combination of the values $p_k, k=\tau,\ldots,d$ of the pmf $p_S$ of $S$:
\begin{equation} \label{eq:momtau}
\sum_{1 \leq i_1<\ldots<i_\tau \leq d} E[X_{i_1}  \cdots X_{i_\tau}]=\sum_{k=\tau}^d \binom{k}{\tau} p_k
\end{equation}
\end{proposition}


\begin{corollary}
Let $f \in \mathcal{F}_d$ be the pmf of a $d$-dimensional multivariate Bernoulli random variable $\XX=(X_1,\ldots,X_d)$ and $p_S$ the pmf of the sum $S=\sum_{i=1}^dX_i$, with $p_S(k)=p_k=P(S=k), k=0,\ldots,d$. If $p_k=0$ for $k\geq 2$ then $E[X_iX_j]=0$ for $1 \leq i <j \leq d$.
\end{corollary}

\begin{corollary}
Given $S^*$, a discrete random variable defined over $\{0,1,\ldots,d\}$ with pmf $p_S$, $p_S(k)=p_k=P(S=k)$, $k=0,\ldots, d$, let $f \in \mathcal{F}_d$ be the pmf of a $d$-dimensional multivariate Bernoulli random variable $\XX=(X_1,\ldots,X_d)\in {\mathcal{A}}^*$,  where ${\mathcal{A}}^*=\{\XX\in\FFF_d: \sum_{i=1}^dX_i=S^*\}$. The average second-order cross moment $\bar{\mu}_2= \frac{\sum_{1 \leq i<j \leq d} E[X_iX_j]}{\binom{d}{2}}$ can be computed as
\[
\bar{\mu}_2=\frac{1}{d(d-1)}\sum_{k=2}^{d} k(k-1) p_{k}
\]
\end{corollary}

We observe that for an exchangeable multivariate Bernoulli random variable the average second-order cross moment $\bar{\mu}_2$ coincides with any second-order cross moment $E[X_iX_j], 1 \leq i <j \leq d$.

\begin{corollary}
If $pd$ is not integer, given $S^*=S_{J^M, j^m} \in \mathcal{D}_d(dp)$,  let $f \in \mathcal{F}_d$ be the pmf of a $d$-dimensional multivariate Bernoulli random variable $\XX=(X_1,\ldots,X_d)\in \mathcal{A}^*$, where ${\mathcal{A}}^*=\{\XX\in\FFF_d: \sum_{i=1}^dX_i=S^*\}$. The average second-order cross moment $\bar{\mu}_2= \frac{\sum_{1 \leq i<j \leq d} E[X_iX_j]}{\binom{d}{2}}$ can be computed as
\begin{equation} \label{eq:mu2star}
\bar{\mu}_2=\frac{1}{d(d-1)} (J^M (2pd-J^M-1))
\end{equation}
If $pd$ is integer we have $\bar{\mu}_2=\frac{1}{d(d-1)} (pd (pd-1))=\frac{1}{(d-1)} (p (pd-1))$
\end{corollary}

\begin{example}
Let $p=11/20$ and $d=5$. We have $pd=11/4=2.75$ and $J^M=2$. Using Equation \ref{eq:mu2star} we get $\bar{\mu}_2=0.25$.  We also have  $S^*=S_{J^M, j^m}\equiv S_{(2, 3)}$. The pmf $\pp_S=(p_0,\ldots,p_5)$ of $S_{2, 3}$ is defined as $p_2=0.25$, $p_3=0.75$ and $p_k=0$,  $k=0,1,4,5$. It follows that all the $5$-dimensional multivariate Bernoulli random variables $\XX=(X_1,\ldots,X_5)$ such that $\sum_{i=1}^5X_i=S_{2, 3}$ have the average of its second-order cross moments equal to $0.25$.
\end{example}

The exchangeable case is more easy because the geometrical representation is simpler than that of the general one. The generators of the exchangeable polytope are known in closed form and in a one-to-one correspondence with the generators of  $\mathcal{D}_d(dp)$.

Using the simple algebraic approach proposed here, we can solve the more challenging problem to  explicitly find a polynomial corresponding to a non exchangeable Bernoulli pmf $\ff\in \chi^*$. This means that if  $\XX$ has pmf $\ff$ then its sum $S=\sum_{i=1}^d$ has support on $j^M, j^m$. The vector $\XX$ correspond to the minimum convex order and to the minimal mean correlation.
Notice that pmfs corresponding to fundamental polynomials  $F_{j_1,\ldots, j_n}(\xx)$ are not in $\chi^*$, for $n\neq d-1$, because the pmfs of  their sums have support on the three points: $\{0, n, d-1\}$. The pmf of the sum corresponding to $F_{d-1}$ has support on $\{0, d-1, d-1\} \equiv \{0, d-1\}$ and then it is not of interest apart from the simple case $d=2$. Similarly, pmfs corresponding to fundamental polynomials  $-F_{j_1,\ldots, j_n}(\xx)$ are not in $\chi^*$, for $n\neq d$, because their sums have support on three points: $\{1, d-n, d\}$.  The pmf of the sum corresponding to $-F_{d-1}$ has support on $\{1, d-(d-1), d\} \equiv \{1, d\}$ and then it is not of interest apart from the simple case $d=2$.

We would like to build a non-exchangeable random variable $X \in \mathcal{F}_d(p)$ whose pmf $\ff_\star$ has support only on the points $x \in \{0,1\}^d$ with $|x|=j^M$ or $|x|=j^m$. In this way the corresponding $\sum_{i=1}^d X_i$ will have support only on $j^M$ and $j^m$.
A possible way for building $\ff_\star$ is based on a well-known tool in algebraic statistics, \cite{diaconis1998algebraic}. We build the exchangeable $\ff_e$ corresponding to the discrete random variable $S_{j^M,j^m}$. This can be easily done taking into consideration Equation \ref{binul} and Equation \ref{onenul}. Then $f_\star$ can be built as $f_e+\epsilon m$, $\epsilon \in (-1,+1)$, where the move $m$ must satisfy the constraints $H\boldsymbol{m}=0$, $\ff_e+\epsilon \boldsymbol{m} \geq 0$, and $\sum_{i=1}^D (\ff_e+\epsilon \boldsymbol{m})_i=1$. The move $\boldsymbol{m}$  can be chosen in the Markov Basis of the matrix $H$. This method works well for small dimensions but become computationally unfeasible for large dimensions $d$.
We now prove that making
 use of the polynomial structure of the generators of $\mathcal{F}_d(p)$ we find a method that provides an analytical solution of the problem and therefore  works also for large dimensions.
  A non-exchangeable pmf $\ff\in \chi^*$ is found as the type-0 pmf associated to a specific linear combination of fundamental polynomials as stated in Theorem \ref{mincx}.

\begin{theorem}\label{mincx}
Let $p=s/t\leq 0.5$,  $a=\frac{2s-t}{s}$, $a_1=|2s-t|=t-2s$ and $a_2=s$. If $p=1/2$ we obtain $a=0$, $a_1=0,\,\,\, a_2={1}$.

\begin{enumerate}
\item We first consider the case  $pd$ not integer.
\begin{enumerate}
\item If $pd+p<j^m$ there are $\alpha_i, \beta_i\in \{0,1\}^d$ with $|\alpha_i|=j^M$, and $|\beta_i|=j^m$ such that the polynomial
\begin{equation*}
P_{d-j^M}(\xx)=-a_2\prod_{i=1}^{d-j^M}x_i+\sum_{i=1,\\ |\alpha_i|=j^M}^{h}\xx^{\alpha_i}+ \sum_{i=1,\\ |\beta_i|=j^m}^{k}\xx^{\beta_{i}}-a_1,
\end{equation*}
where
\begin{equation*}\begin{split}
&k=a_2d-2a_2j^M-a_1j^M\\
&h=a_1+a_2-k.
\end{split}
\end{equation*}
belongs to the ideal $I_\mathcal{P}$ and the corresponding $\XX\in \FF_d(p)$  has   sum $S^*=\sum_{i=1}^dX_i$  with support on $\{j^M, j^m\}$.

\item If $pd+p\geq j^m$, there are $\alpha_i, \beta_i\in \{0,1\}^d$ with $|\alpha_i|=j^M$, and $|\beta_i|=j^m$ such that  the polynomial
\begin{equation*}
P_{d-j^m}(\xx)=-a_2\prod_{i=1}^{d-j^m}x_i+ \sum_{i=1,\\ |\alpha_i|=j^M}^{h}\xx^{\alpha_i}+ \sum_{i=1,\\ |\beta_i|=j^m}^{k}\xx^{\beta_{i}}-a_1,
\end{equation*}
where
\begin{equation*}\begin{split}
&k=a_2d-2a_2j^M-a_1j^M-a_2\\
&h=a_1+a_2-k.
\end{split}
\end{equation*}
belongs to the ideal $I_\mathcal{P}$ and the corresponding $\XX\in \FFF_d(p)$ has sum $S^*=\sum_{i=1}^dX_i$  with support on $\{j^M, j^m\}$.
\end{enumerate}
\item If $pd$ is integer the polynomial
\begin{equation*}
P_{pd}(\xx)=-a_2\prod_{i=1}^{d-pd}x_i+ \sum_{i=1,\\ |\alpha_i|=j^M}^{a_1+a_2}\xx^{\alpha_i}-a_1,
\end{equation*}
belongs to the ideal $I$ and the corresponding $\XX\in \FFF_d(p)$ has sum  with support on $\{pd\}$.
\end{enumerate}
\end{theorem}
\begin{proof}
Assume $pd$ not integer. To simplify the notation let $j^M=m$. We have $j^m=m+1$ and $pd\in(m,m+1)$.
The basic idea is to build $P(\xx)$, a linear combination of fundamental polynomials (which belongs to the ideal $I_\mathcal{P}$ by construction) and that could be rewritten as a polynomial $P_{d-m}(\xx)$ (or $P_{d-(m+1)}(\xx)$) whose corresponding type-0 pmf has support only on points $\xx$, with  $|\xx|=m$ or  $|\xx|=m+1$.

We consider two different cases, $pd+p<m+1$ and $pd+p\geq m+1$.

Let $pd+p<m+1$ and let $P(\xx)$ be  the following linear combination of fundamental polynomials
\begin{equation*}
P(\xx)=-a_2F_{d-m}(\xx)+\sum_{i=1,\\ |\alpha_i|=m}^{h} F_{\alpha_i}(\xx)+ \sum_{i=1,\\ |\alpha_i|=m+1}^{k} F_{\beta_i}(\xx),
\end{equation*}
where $F_{\alpha_i}(\xx) \equiv F_{i_1,\ldots,i_n}(\xx)$, $i_1, \ldots, i_n$ are the positions where $\alpha_i$ is one, and $F_{\beta_i}(\xx)$ is similarly defined. We recall that $F_{i_1,\ldots,i_n}(\xx)=x_{i_1} \cdots x_{i_n} - \sum_{j=1}^n x_{i_j} + (n-1)$. Clearly, $P(\xx)\in I$.

We define $P_{d-m}(\xx)=-a_2\prod_{i=1}^{d-m}x_i+\sum_{i=1,\\ |\alpha_i|=m}^{h}\xx^{\alpha_i}+ \sum_{i=1,\\ |\beta_i|=m+1}^{k}\xx^{\beta_{i}}-a_1$ and we look for $h, k$, $\alpha_i, \beta_i$ such that $P(\xx)=P_{d-m}(\xx)$. It follows that:
\begin{enumerate}
\item \label{c1} the constant term of $P_{d-m}(\xx)$, that is $-a_1$ must be equal to the constant term of $P(\xx)$, that is $-a_2(d-m-1)+h(m-1)+km$;
\item \label{c2} because $P_{d-m}(\xx)$ has not linear terms, the linear terms of $P(\xx)$ must be zero. It follows that the (positive) linear terms of $-a_2F_{d-m}(\xx)$ have to be cancelled by the (negative) linear terms of $F_{\alpha_i}(\xx)$ and $F_{\beta_i}(\xx)$. It is worth noting that in this proof given the term $\gamma_i x_i$, $\gamma_i$ integer, we say that there are $\gamma_i$ linear terms, i.e. we consider $\gamma_i x_i= \underbrace{x_i+\ldots+x_i}_{\gamma_i \text{times}}$ if $\gamma_i>0$ and $\gamma_i x_i= \underbrace{-x_i -\ldots- x_i}_{\gamma_i \text{times}}$ if $\gamma_i<0$.
\end{enumerate}

Condition \ref{c1} is satisfied for any choice of $\alpha_i, i=1,\ldots,k, |\alpha_i|=m$, and $\beta_i, i=1,\ldots,k,  |\beta_i|=m+1$ if $h$, and $k$ are positive solutions of:
\begin{equation*}
h(m-1)+km=a_2(d-m-1)-a_1.
\end{equation*}

To satisfy Condition \ref{c2} we first look for h, k such that the total number of linear terms in  $F_{\alpha_i}(\xx)$ and $F_{\beta_i}(\xx)$ is equal to the total number of linear terms in $-a_2F_{d-m}(\xx)$.  Then we will show that by properly choosing the $h$ polynomials $F_{\alpha_i}$ of degree $m$ and the $k$ polynomials $F_{\beta_i}$ of degree $m+1$ we can simplify all the linear terms in $P(\xx)$. Since all $F_{\alpha_i}(\xx)$  have the same number $m$ of linear terms for any $i$, all $F_{\beta_i}(\xx)$ have the same number $m+1$ of linear terms for any $i$, and the number of linear terms of $-a_2F_{d-m}(\xx)$ is $a_2(d-m)$,  $h$ and $k$ must be positive solutions of:
\begin{equation*}
hm+k(m+1)=a_2(d-m).
\end{equation*}

From Conditions  \ref{c1} and  \ref{c2} we have to find the solutions of
\begin{equation*}\begin{cases}
h(m-1)+km=a_2(d-m-1)-a_1\\
hm+k(m+1)=a_2(d-m).
\end{cases}
\end{equation*}
Standard computations give
\begin{equation*}\begin{split}
&k=a_2d-2a_2m-a_1m\\
&h=a_1+a_2-k.
\end{split}
\end{equation*}
We must check that solutions are positive integers. It is possible to verify that the solutions are integer since
$(m-1)(m+1)-m^2=-1$ and that the solutions are positive iff $pd+p<m+1=j^m$.

A possible not unique choice for $\alpha_i$ and $\beta_i$ can be obtained using the following steps.
\begin{enumerate}
\item The $a_2$ copies of the linear terms $x_1, \ldots, x_{d-m}$ must be ordered, repeating the sequence "$x_1, \ldots, x_{d-m}$" $a_2$ times
\begin{equation} \label{list}
\underbrace{x_1, \ldots, x_{d-m}}_{\text{1st time}}, \underbrace{x_1, \ldots, x_{d-m}}_{\text{2nd time}}, \ldots, \underbrace{x_1, \ldots, x_{d-m}}_{a_2\text{-th time}}
\end{equation}
\item The $\alpha_i, i=1,\ldots,h,\; |\alpha_i|=m$ are determined. This is equivalent to build $h$ mononomials of degree $m$. The first monomial will be built as the product of the first $m$ linear terms (i.e. variables) in the list shown in Equation \ref{list}, i.e. $x_1\cdots x_m$, the second monomial with the subsequent $m$ variables, i.e. $x_{m+1}\cdots x_{2m(mod(d-m))+1}$, and so on for the first $h$ mononomials.
\item Then, in an analogous way, starting from the position $hm+1$ to the end of the list shown in Equation \ref{list} we build $k$ monomials of degree $m+1$. 
\end{enumerate}



Let's now consider the case $pd+p\geq j^m$.  
The proof is similar to the case $pd+p> j^m$ and we look for the positive solutions of
\begin{equation*}\begin{cases}
h(m-1)+km=a_2(d-m-2)-a_1\\
hm+k(m+1)=a_2(d-m-1).
\end{cases}
\end{equation*}
Easy computations give
\begin{equation*}\begin{split}
&k=a_2d-2a_2m-a_1m-a_2\\
&h=a_1+a_2-k.
\end{split}
\end{equation*}
If $pd+p\geq m+1=j^m$ then $pd-m\geq 1-p$ and since $p<1/2$, $pd-m>p$. It is easy to verify that if $pd-m>p$, $h$ and $k$  are both  integers and  positives. Notice that $p=1/2$ implies  $pd+p\geq j^m$ for $d\geq2$, thus it is included in this case.

The case $pd$ integer can be proved similarly by proving that $P_{pd}(\xx)$ is the following linear combination of fundamental polynomials:
\begin{equation*}
P_{pd}(\xx)=-a_2F_{d-pd}(\xx)+ \sum_{i=1,\\ |\alpha_i|=pd}^{a_1+a_2}   F_{\alpha_i}(\xx).
\end{equation*}

\end{proof}

The proof of Proposition \ref{mincx} provides a way to find $\alpha_i$ and $\beta_i$ with a very simple algorithm. Therefore we can easily find a pmf $\ff\in \FFF_d(p)$ minimal with respect to the convex order in any dimension $d$ and for any $p$. From  Lemma 2.3 in \cite{terzer2009large}, we can also check if the density found is an extremal point - we know that the minimum is reached on at least one extremal point.

We conclude this section with some  examples.  Example \ref{E10} of $P_{d-j^M}(\xx)$ in a case where $pd+p<j^m$. Examples \ref{E11} and \ref{E14} show the polynomials corresponding to $\XX\in \chi^*$ in high dimensions and with different choices for $p$.
Example \ref{E13} shows the case $pd$ integer. In all cases, given the final polynomial, the corresponding type-0 pmf can be easily found using the steps in the algorithm described in Section \ref{AlgoPol}.

\begin{example}\label{E10}
Let $d=7$, $s=2$, and $t=5$. We have $p=s/t=2/5$, $pd=14/5=2.8$, $j^M\equiv m=2$, $j^m\equiv m+1=3$, $a=(2s-t)/s=-1/2$, $a_1=1$, and $a_2=2$.

\begin{table}[h]
	\centering
		\begin{tabular}{r|r|r||r|r|r|r}
$S$ &	$\xx_{\alpha_i}$ &	constant &$d_j=d-j$&	$\xx_{d_j}$ & 		constant & 		a \\
		\hline
$S=2$&	$x_{i_1}x_{i_2}$&+1&5&$-x_1\cdots x_5$&-4&$-\frac{1}{2}$\\
$S=3$&	$x_{i_1}x_{i_2}x_{i_3}$&+2&4&$-x_1\cdots x_4$&-3&$-\frac{1}{2}$\\
		\end{tabular}
	\caption{Terms to be balanced}
\end{table}
Since $pd+p=\frac{16}{5}>3$, we have to consider $P_{d-j^m}(\xx) \equiv P_{7-3}(\xx) = P_{4}(\xx)$. Thus we have to find $h, k$ such that
\begin{equation*}\begin{cases}
h+2k=5\\
2k+3k=8.
\end{cases}
\end{equation*}
We find $h=1, k=2$.

Since the linear term of $-2F_{1,\ldots,4}(\xx)$ is $2\sum_{i=1}^4x_i$ we have split the $2\cdot4=8$ variables listed in the first row of Table \ref{tab:x} using $h=1$ group with $m=2$ variables and $k=2$ groups with $m+1=3$ variables. The second row of Table \ref{tab:x} reports the corresponding monomials.

\begin{table}[h]
	\centering
		\begin{tabular}{r|r|r|r|r}
Linear terms&$x_1\,\,\,	x_2$ &	$x_3	\,\,\,x_4 \,\,\, x_1$&	$x_2\,\,\,x_3 \,\,\,	x_4$	 \\
\hline
Monomials &$x_1x_2$&$x_3x_4x_1$&${x_2x_3x_4}$
		\end{tabular}
	\caption{Linear terms}
	\label{tab:x}
\end{table}
The resulting polynomial is:
\begin{equation*}
P_{4}(\xx)=-2x_1x_2x_3x_4+x_1x_2+x_1x_3x_4+x_2x_3x_4-1.
\end{equation*}

\end{example}
\begin{example}\label{E11}
Consider the following two cases:
\begin{enumerate}
\item Let $d=9$, $s=2$, and $t=5$. We have $p=s/t=2/5$, $pd=18/5=3.6$, $j^M\equiv m=3$, $j^m\equiv m+1=4$, $a=(2s-t)/s=-1/2$, $a_1=1$, and $a_2=2$.
Since $pd+p=\frac{20}{5}=4$, we have to consider $P_{d-j^m}(\xx)=P_{5}(\xx)$.
We find $h=2, k=1$ and
\begin{equation*}
P_5(\xx)=-2x_1\cdots x_5+x_1x_2x_3+x_1x_4x_5+x_2x_3x_4x_5-1.
\end{equation*}

\item Let $d=9$, $s=2$, and $t=7$. We have $p=s/t=2/7$, $pd=18/7 \approx 2.57$, $j^M\equiv m=2$, $j^m\equiv m+1=3$, $a=(2s-t)/s=-3/2$, $a_1=3$, and $a_2=2$.
Since $pd+p=\frac{20}{7}<3$, we have to consider $P_{d-m}(\xx)=P_{7}(\xx)$.
We find $h=1, k=4$ and
\begin{equation*}
P_7(\xx)=-2x_1\cdots x_7+x_1x_2+x_3x_4x_5+x_1x_6x_7+x_2x_3x_4+x_5x_6x_7-3.
\end{equation*}

\end{enumerate}
\end{example}

\begin{example}\label{E13}
Let $d=5$, $s=2$, and $t=5$. We have $p=s/t=2/5$, $pd=2$ $a=(2s-t)/s=-1/2$, $a_1=1$, and $a_2=2$. The value of $pd$ is integer and then we ave to consider $P_{d-pd}(\xx)=P_{3}(\xx)$.
\begin{equation*}
P_{3}(\xx)=-2x_1 x_2 x_3 +x_1 x_2 + x_{1} x_{3} + x_{2} x_{3}  -1.
\end{equation*}
\end{example}
\begin{example}\label{E14}
Let $d=216$, $s=2$, and $t=5$. We have $p=s/t=2/5$, $pd=432/5 =86.4$, $j^M\equiv m=86$, $j^m\equiv m+1=87$, $a=(2s-t)/s=-1/2$, $a_1=1$, and $a_2=2$.
Since $pd+p=86.8$ is less than $m+1=87$, we have to consider $P_{d-m}(\xx)=P_{130}(\xx)$.
We find $h=1, k=2$ and
\begin{equation*}
P_{130}(\xx)=-2x_1\cdots x_{130}+x_1 \cdots x_{86}+ x_{1} \cdots x_{43} \cdot x_{87} \cdots x_{130}+ x_{44} \cdots x_{130}  -1.
\end{equation*}
\end{example}

\subsection{The safest dependence structure}

Since convex order is a variability order the minimal convex order correspond to a minimal risk random variable and the minimal convex order can be thought as the safest dependence structure. If $pd<1$ the latter is linked to the strongest negative dependence, i.e. mutual exclusivity. This section generalizes the notion of mutual exclusivity to the case $pd>1$ and discusses its link with the minimal convex order.

In \cite{cheung2014characterizing} the authors characterize mutual exclusivity as the strongest negative dependence structure. When mutual exclusivity is possible, it corresponds to the minimal convex order and in this light it can be considered the safest dependence structure.
A vector $\XX\in \FFF_d(p)$ is mutually exclusive if

\begin{equation*}
P(X_{i_1}=1, X_{i_2}=1)=0, \,\,\, \forall i_1, i_2\in\{0,\ldots,d\}.
\end{equation*}
The authors also show that the distribution of a mutually exclusive random vector is the lower Fr\'echet bound of its Fr\'echet class, that in our framework is $F_L(\xx)=max(\sum_{i=1}^dF_i(x_i)-d+1, 0)$.
 Nevertheless, if $pd>1$ mutual exclusivity is not possible and the lower Fr\'echet bound is not a pmf.
We now generalize the notion of mutual exclusivity to any $p$ and $d$ and discuss its link with negative correlation.

%
%

\begin{definition}
Let $\XX\in \FFF_d(p)$, $\XX$ is mutually exclusive of order $m$ if

\begin{equation*}
P(X_{i_1}=1,\ldots, X_{i_m}=1)=0, \,\,\, \forall i_1,\ldots, i_m\in\{0,\ldots,d\}.
\end{equation*}
\end{definition}

If $m=2$, $\XX$ is mutually exclusive.

\begin{proposition}
Let $\XX\in \FFF_d(p)$ then $\XX$ is  mutually exclusive of order $m$ iff $P(S\geq m)=0$.
\end{proposition}
\begin{proof}
By contradiction, if $P(S\geq m)\neq0$ then  $P(S=k)\neq 0$ for a $k> m$ and  it exists $X_{i_1}, \ldots, X_{i_k}$ such that $P(X_{i_1}=1,\ldots, X_{i_k}=1)\neq0$, therefore $\XX$ is not $k$-exclusive. $P(X_{i_1}=1,\ldots, X_{i_m}=1)\geq P(X_{i_1}=1,\ldots, X_{i_k}=1)$ since $m\leq k$ thus $\XX$ is not $m$- exclusive.

If $\XX$ is $m$-exclusive then $P(S=m)=0$  and $P(S= k)\leq P(S=m)$ for $k\geq m$ thus $P(S\geq m)=0$.
\end{proof}

The following result is a straightforward consequence of the fact that, to preserve the condition on the mean,
if  $dp\in(j^M, j^m]$, then $P(S\geq j^m)\neq 0$.

\begin{proposition}
Let $\XX\in \FFF_d(p)$. If $dp\in (j^M, j^m]$ then $\XX$ cannot be $k$-exclusive for all $k\leq j^m$.
\end{proposition}

This is in line with the condition that the lower Fr\'echet bound is a distribution and belongs to the  class $\FFF_d(p)$ iff $pd<1$. 
A consequence of the above results is that if $pd\in(j^M, j^m]$ the vector in $\FFF_d(p)$ cannot  be mutually exclusive of order lower  that $j^m$.
Thus $j^m$-mutually exclusive pmfs have sums with support on $\{0,1,\ldots, j^m\}$. Since the minimum of the  convex order is reached on $\chi^*$, the safest dependence structure is $j^m$-mutually exclusive. Nevertheless, Proposition \ref{min:cx} implies that if $pd>1$ not all the $j^m$-mutually exclusive pmfs are minimal with respect to the convex order, but only the ones with support on $\{j^M, j^m\}$.

{As already observed, fundamental polynomials of degree $n$ and their opposite have  sums with support on three points: $\{0,n, d-1\}$ or $\{1, d-n, d\}$, respectively. This means that they cannot be $m$-exclusive of order lower that $m=d-1$ or $m=d$.
The peculiarity of their support let us wonder whether they are associated to specific dependence structures that are able to generate the admissible dependencies in the class.}

\section{Conclusion}\label{Conc}

We map the class $\FFF_d(p)$ into an ideal of points and we show that in any dimension the class $\FFF_d(p)$ is generated by a set of polynomials that we call fundamental polynomials.
Each pmf in $\FFF_d(p)$ is associated to a linear combination of fundamental polynomials.
This representation turns out to be very important to address open issues in the study of multivariate Bernoulli distributions.
As a first application, we prove that a specific linear combination of fundamental polynomials solves the open problem to find a minimal distribution with respect to the convex order of sums and with respect to a measure of negative dependence.
 Lower bounds in the convex order identify the safest dependence structure, that is important in many fields, such as insurance or finance.
In particular we are interested in applying our results to credit portfolio management,  where multivariate  Bernoulli distributions are used to model indicators of default.  In this framework sums represent aggregate defaults and their bounds in the convex order are helpful to identify bounds for the risk associated to credit portfolios.  Nowadays, the real economy is highly interconnected and  banks and financial intermediaries are exposed to  losses arising from defaults of obligors that are not independent.
 Since there are usually hundreds of obligors, banks need to handle high dimensional portfolios, hence the importance of  analytical results.

Our theoretical future research will focus on fundamental polynomials, as generators of the $\FFF_d(p)$, and in particular on two open issues.
First, fundamental polynomials of degree $n$, $n<d-1$ and their opposites have  sums with support on three points: $0,n, d-1$ or $1, d-n, d$, respectively. This means that they cannot be $m$-exclusive of order lower than $m=d-1$ or $m=d$.
The peculiarity of their support and its connection with a dependence notion puts in question if pmfs associated to fundamental polynomials can be characterized in terms of their dependence structure and if we can then identify a class of dependencies  that are able to generate all the admissible dependencies in the class.
Second, fundamental polynomials correspond to extremal points of the polytope $\FFF_d(p)$  and we have used them in this work to provide an algorithm to  find  extremal points in high dimensions. The connection between algebraic and geometrical generators and also their connections with bounds for the class will be  further investigated.
\appendix

\section*{Acknowledgements}
The authors gratefully acknowledge financial support from the Italian Ministry of Education, University and Research, MIUR, "Dipartimenti di Eccellenza"
grant 2018-2022.

\bibliographystyle{apalike}
\bibliography{biblioRF}

\begin{thebibliography}{}

\bibitem[Chaganty and Joe, 2006]{chaganty2006range}
Chaganty, N.~R. and Joe, H. (2006).
\newblock Range of correlation matrices for dependent {B}ernoulli random
  variables.
\newblock {\em Biometrika}, 93(1):197--206.

\bibitem[Cheung and Lo, 2014]{cheung2014characterizing}
Cheung, K.~C. and Lo, A. (2014).
\newblock Characterizing mutual exclusivity as the strongest negative
  multivariate dependence structure.
\newblock {\em Insurance: Mathematics and Economics}, 55:180--190.

\bibitem[Cox et~al., 2013]{cox2013ideals}
Cox, D., Little, J., and OShea, D. (2013).
\newblock {\em Ideals, varieties, and algorithms: an introduction to
  computational algebraic geometry and commutative algebra}.
\newblock Springer Science \& Business Media.

\bibitem[Dai et~al., 2013]{dai2013multivariate}
Dai, B., Ding, S., and Wahba, G. (2013).
\newblock Multivariate bernoulli distribution.
\newblock {\em Bernoulli}, 19(4):1465--1483.

\bibitem[Denuit and Vermandele, 1998]{denuit1998optimal}
Denuit, M. and Vermandele, C. (1998).
\newblock Optimal reinsurance and stop-loss order.
\newblock {\em Insurance: Mathematics and Economics}, 22(3):229--233.

\bibitem[Dhaene and Denuit, 1999]{dhaene1999safest}
Dhaene, J. and Denuit, M. (1999).
\newblock The safest dependence structure among risks.
\newblock {\em Insurance: Mathematics and Economics}, 25(1):11--21.

\bibitem[Diaconis and Sturmfels, 1998]{diaconis1998algebraic}
Diaconis, P. and Sturmfels, B. (1998).
\newblock Algebraic algorithms for sampling from conditional distributions.
\newblock {\em The Annals of statistics}, 26(1):363--397.

\bibitem[Emrich and Piedmonte, 1991]{emrich1991method}
Emrich, L.~J. and Piedmonte, M.~R. (1991).
\newblock A method for generating high-dimensional multivariate binary
  variates.
\newblock {\em The American Statistician}, 45(4):302--304.

\bibitem[Fontana et~al., 2020]{fontana2020model}
Fontana, R., Luciano, E., and Semeraro, P. (2020).
\newblock Model risk in credit risk.
\newblock {\em Mathematical Finance}, pages 1--27.

\bibitem[Fontana and Semeraro, 2018]{fontana2018representation}
Fontana, R. and Semeraro, P. (2018).
\newblock Representation of multivariate {B}ernoulli distributions with a given
  set of specified moments.
\newblock {\em Journal of Multivariate Analysis}, 168:290--303.

\bibitem[Fontana and Semeraro, 2021]{fontana2021exchangeable}
Fontana, R. and Semeraro, P. (2021).
\newblock Exchangeable bernoulli distributions: high dimensional simulation,
  estimate and testing.
\newblock {\em arXiv preprint arXiv:2101.07693}.

\bibitem[Haynes et~al., 2016]{haynes2016simulating}
Haynes, M.~E., Sabo, R.~T., and Chaganty, N.~R. (2016).
\newblock Simulating dependent binary variables through multinomial sampling.
\newblock {\em Journal of Statistical Computation and Simulation},
  86(3):510--523.

\bibitem[Kaas et~al., 2000]{kaas2000upper}
Kaas, R., Dhaene, J., and Goovaerts, M.~J. (2000).
\newblock Upper and lower bounds for sums of random variables.
\newblock {\em Insurance: Mathematics and Economics}, 27(2):151--168.

\bibitem[Kang and Jung, 2001]{kang2001generating}
Kang, S.-H. and Jung, S.-H. (2001).
\newblock Generating correlated binary variables with complete specification of
  the joint distribution.
\newblock {\em Biometrical Journal}, 43(3):263--269.

\bibitem[Kvam, 1996]{kvam1996maximum}
Kvam, P.~H. (1996).
\newblock Maximum likelihood estimation and the multivariate bernoulli
  distribution: An application to reliability.
\newblock In {\em Lifetime Data: Models in Reliability and Survival Analysis},
  pages 187--194. Springer.

\bibitem[Marchetti et~al., 2016]{marchetti2016palindromic}
Marchetti, G.~M., Wermuth, N., et~al. (2016).
\newblock Palindromic {B}ernoulli distributions.
\newblock {\em Electronic Journal of Statistics}, 10(2):2435--2460.

\bibitem[McNeil et~al., 2005]{mcneil2005quantitative}
McNeil, A.~J., Frey, R., and Embrechts, P. (2005).
\newblock {\em Quantitative risk management}, volume~3.
\newblock Princeton {U}niversity {P}ress.

\bibitem[Padmanabhan and Natarajan, 2021]{padmanabhan2021tree}
Padmanabhan, D. and Natarajan, K. (2021).
\newblock Tree bounds for sums of bernoulli random variables: A linear
  optimization approach.
\newblock {\em Informs Journal on Optimization}, 3(1):23--45.

\bibitem[Qaqish, 2003]{qaqish2003family}
Qaqish, B.~F. (2003).
\newblock A family of multivariate binary distributions for simulating
  correlated binary variables with specified marginal means and correlations.
\newblock {\em Biometrika}, 90(2):455--463.

\bibitem[Shaked and Shanthikumar, 2007]{shaked2007stochastic}
Shaked, M. and Shanthikumar, J.~G. (2007).
\newblock {\em Stochastic orders}.
\newblock Springer.

\bibitem[Shults, 2016]{shults2016simulating}
Shults, J. (2016).
\newblock Simulating longer vectors of correlated binary random variables via
  multinomial sampling.

\bibitem[Terzer, 2009]{terzer2009large}
Terzer, M. (2009).
\newblock {\em Large scale methods to enumerate extreme rays and elementary
  modes}.
\newblock PhD thesis, ETH Zurich.

\end{thebibliography}

\end{document}